\documentclass[a4paper,10pt]{amsart}

\usepackage[T1]{fontenc}
\usepackage[utf8]{inputenc}
\usepackage{amsmath,amsthm,amssymb,bbm,enumerate,mathrsfs,tabularx}
\usepackage[all]{xy}

\theoremstyle{plain}
\newtheorem{lem}{Lemma}[section]
\newtheorem{thm}[lem]{Theorem}
\newtheorem{prop}[lem]{Proposition}
\newtheorem{ex}[lem]{Example}
\theoremstyle{definition}
\newtheorem{defn}[lem]{Definition}
\newtheorem{rem}[lem]{Remark}
\newtheorem{set}[lem]{Setting}
\newtheorem{notation}[lem]{Notation}

\usepackage[english]{babel}

\DeclareMathOperator{\intw}{\mathcal Intw}
\DeclareMathOperator{\Hom}{Hom}
\DeclareMathOperator{\irr}{irr}
\DeclareMathOperator{\End}{End}
\DeclareMathOperator{\id}{Id}

\numberwithin{equation}{subsection}

\date{\today}

\title[Morita reduced versions of skew group algebras of path
algebras]{On the Morita reduced versions of skew group algebras of path
  algebras}

\author{Patrick Le Meur}

\address{Universit\'e de Paris, Sorbonne Universit\'e, CNRS, Institut
  de Math\'ematiques de Jussieu-Paris Rive Gauche, F-75013 Paris,
  France}

\email{patrick.le-meur@imj-prg.fr}



\begin{document}

\begin{abstract}
  Let $R$ be the skew group algebra of a finite group acting on the
  path algebra of a quiver. This article develops both theoretical and
  practical methods to do computations in the Morita reduced algebra
  associated to $R$. Reiten and Riedtmann proved that there exists an
  idempotent $e$ of $R$ such that the algebra $eRe$ is both Morita
  equivalent to $R$ and isomorphic to the path algebra of some quiver
  which was described by Demonet. This article gives explicit formulas
  for the decomposition of any element of $eRe$ as a linear
  combination of paths in the quiver described by Demonet. This is
  done by expressing appropriate compositions and pairings in a
  suitable monoidal category which takes into account the
  representation theory of the finite group.
\end{abstract}

\maketitle

\section*{Introduction}
\label{sec:introduction}

In \cite{RR85}, Reiten and Riedtmann initiated the investigation of
the skew group algebras of Artin algebras from the viewpoint of
homological dimensions and representation theory. Their article was
followed up by many research works illustrating the following
principle: Artin algebras share many properties, whether of
homological or of representation theoretic nature, with their
associated skew group algebras. For instance, see \cite{MR1913868} for
interactions with Koszul duality; \cite{MR2467401,MR3395330} for
interactions with piecewise hereditary algebras; and
\cite{MR978602,MR2578593} for interactions with preprojective algebras
and McKay correspondence.

Let ${\mathbbm{k}}$ be an algebraically closed field, $Q$ be a finite quiver
that may have oriented cycles, and $G$ be a finite group with order
not divisible by $\mathrm{char}({\mathbbm{k}})$ and acting on ${\mathbbm{k}} Q$ by algebra automorphisms
in such a way that both the set of (primitive idempotents of) vertices
and the vector subspace generated by the arrows are stabilised by this
action. Following \cite{RR85}, the skew group algebra ${\mathbbm{k}} Q*G$ is
hereditary.

While Reiten and Riedtmann described a quiver whose path algebra is
Morita equivalent to ${\mathbbm{k}} Q*G$ when $G$ is cyclic, the description of
such a quiver in general is recent. In \cite{MR2578593}, Demonet
described an idempotent $\tilde e$ of ${\mathbbm{k}} Q * G$ and a quiver $Q_G$
such that $\tilde e \cdot ({\mathbbm{k}} Q*G)\cdot \tilde e$ is Morita equivalent
to ${\mathbbm{k}} Q*G$ and isomorphic to ${\mathbbm{k}} Q_G$. This, however, does not yield
a completely described isomorphism, which might be source of
trouble. Here is a situation taken from Calabi-Yau algebras and where
such trouble may occur. Following \cite{lemeur2}, if $W$ is a
$G$-invariant potential on $Q$, then $G$ acts on the Ginzburg dg
algebra $\mathcal A(Q,W)$, defined in \cite{G06}, and
$\mathcal A(Q,W)*G$ is Morita equivalent to $\mathcal A(Q_G,W_G)$,
where $W_G$ is the potential on $Q_G$ which, as a linear combination
of oriented cycles, corresponds to $\tilde e \cdot W \cdot \tilde e$
under any chosen isomorphism
${\mathbbm{k}} Q_G \to \tilde e \cdot( {\mathbbm{k}} Q*G) \cdot \tilde e$. The problem here
is that describing that linear combination explicitly is not easy. Up
to now, the only fairly general explicit descriptions of $W_G$ are due
to Amiot and Plamondon \cite{AP} when $G$ is of order $2$ and to
Giovannini and Pasquali \cite{2018arXiv180504041G} when $G$ is
cyclic, the stabiliser of each vertex is either trivial or the whole
group, and every cycle appearing in $W$ goes through fixed vertices
only (or, through vertices with trivial stabilisers only,
respectively) as soon as it goes through two of them; during the
revision of the present article, Giovannini, Pasquali and Plamondon
presented \cite{GPP} an extension of the latter description to the
case where $G$ is finite and abelian.

This raises the question of decomposing explicitly the elements of
$\tilde e \cdot ({\mathbbm{k}} Q*G)\cdot \tilde e$ as linear combinations of
paths in $Q_G$ under an isomorphism
${\mathbbm{k}} Q_G \to \tilde e \cdot( {\mathbbm{k}} Q*G) \cdot \tilde e$.

Demonet's description of $Q_G$ involves the representation theory of
the stabilisers of the vertices of $Q$. It is hence expectable that
answering the above-mentioned question should involve representation
theory. Answers of this kind already exist in particular cases where
$Q$ has only one vertex. Here is an example due to Ginzburg
\cite{G06}. He proved that, if $G$ is a finite subgroup of
$\mathrm{SL}_3(\mathbb C)$, then $\mathbb C[x,y,z]*G$ is Calabi-Yau in
dimension $3$. This was done by taking $Q$ to be the quiver with one
vertex and three loops $x$, $y$, and $z$, by taking $W=xyz-xzy$, and
by expressing $W_G$ in terms of the monoidal category of finite
dimensional representations of $G$. See \cite{MR2593679} for a
generalisation to $\mathbb C[x_1,\ldots,x_n]*G$ where $G$ is a finite
subgroup of $\mathrm{SL}_n(\mathbb C)$, in which case $Q$ is the
quiver with one vertex and $n$ loops. Note that this generalisation is
(detailed and) successfully applied by Giovannini in \cite{giovannini}
to find subgroups $\Gamma$ of $\mathrm{SL}_{n+1}(\mathbb C)$ such that
$\mathbb C[x_1,\ldots,x_{n+1}]*\Gamma$ is Morita equivalent to the
$(n+1)$-preprojective algebra of an $n$-representation finite algebra.

This article hence gives explicit formulas for the decomposition of
the elements of $\tilde e \cdot ({\mathbbm{k}} Q*G)\cdot \tilde e$ as linear
combinations of paths in $Q_G$ under the isomorphism
${\mathbbm{k}} Q_G \to \tilde e \cdot( {\mathbbm{k}} Q*G) \cdot \tilde e$. These formulas
are obtained by translating explicitly the result of natural
operations and pairings in the monoidal category
$(\mathrm{mod}(A^e),\otimes_A)$ of finite dimensional $A$-bimodules, where $A$
is the direct product of the group algebras of the stabilisers of the
vertices of $Q$.

On one hand, these formulas provide an interpretation of the
coefficients of these linear combinations in terms of the
representation theory of the stabilisers of the vertices of $Q$. On
the other hand, they are explicit enough to be implemented by a
computer or a human being able to manipulate the irreducible
representations of the involved groups. In specific cases, for
instance when the action of the group elements transforms arrows to
scalar multiples of arrows and the stabilisers are cyclic, these
formulas simplify in a combinatorial way. Finally, these formulas give
a complete solution to the problem of computing a potential $W_G$ on
$Q_G$ such that the Ginzburg dg algebra $\mathcal A(Q_G,W_G)$ is
Morita equivalent to $\mathcal A(Q,W)*G$, for any given $G$-invariant
potential $W$ on $Q$.

\section{Definitions and main results}
\label{sec:defin-notat-main}

Throughout the article, ${\mathbbm{k}}$ denotes an algebraically closed
field. For all finite dimensional algebras $\Lambda$, their categories
of finite dimensional left modules are denoted by $\mathrm{mod}(\Lambda)$.

This section presents the main results of this text as well as the
needed definitions and notation. These definitions are illustrated
with a running example throughout the section. A full application of
the main results is presented on an example in
section~\ref{sec:examples}. The remaining sections are devoted to the
proofs of these results.

\begin{set}
  \label{sec:defin-notat-main-2}
  Let $Q_0$, $S$, $G$, and $M$ be as follows.
  \begin{itemize}
  \item $Q_0$ is a finite set and $S$ is the semi-simple ${\mathbbm{k}}$-algebra
    ${\mathbbm{k}}{Q_0}$.
  \item $G$ is a finite group with order not divisible by $\mathrm{char}({\mathbbm{k}})$ and
    acting on $Q_0$.
  \item $M$ is a finite dimensional $S$-bimodule, considered as a
    collection of vector spaces $(\,_iM_j)_{(i,j)\in Q_0^2}$, endowed
    with an action of $G$ written exponentially and such that
    $^g( \,_iM_j) = \,_{g\cdot i}M_{g\cdot j}$ for all $i,j\in Q_0$
    and $g\in G$.
  \end{itemize}
\end{set}

The choice of a basis of $_iM_j$, for all $i,j\in Q_0$ determines a
quiver $Q$ such that $T_S(M) \simeq {\mathbbm{k}} Q$ as ${\mathbbm{k}}$-algebras. All finite
quivers arise in this way.

The actions of $G$ on $S$ and $M$ induce an action on $T_S(M)$ by
algebra automorphisms such that
$^g(m_1\otimes\cdots \otimes m_n) = \,^gm_1 \otimes \cdots \otimes
\,^gm_n$
for all $g\in G$ and $m_1,\ldots,m_n \in M$. Recall that the skew
group algebra $T_S(M)*G$ is the ${\mathbbm{k}}$-algebra with underlying vector
space $T_S(M) \otimes G$, where the tensor $u \otimes v$ is denoted by $u*v$
for all $u\in T_S(M)$ and $v\in {\mathbbm{k}} G$, and with product such that
$(u * g) \cdot (u'*g') = (u\cdot\,^gu')*gg'$ for all $u,u'\in T_S(M)$
and $g,g'\in G$. It is customary to consider $T_S(M)$ and ${\mathbbm{k}} G$ as
subalgebras of $T_S(M)*G$ in the canonical way.

\begin{notation}
  \label{sec:defin-notat-main-3}
  The following data is used throughout the article.
  \begin{itemize}
  \item  $\left[G\backslash Q_0\right]$ denotes a complete set of
    representatives of the orbits in $Q_0$.
  \item For all $i\in Q_0$,
    \begin{itemize}
    \item $G_i$ denotes the stabiliser of $i$ and $e_i$ denotes the
      corresponding primitive idempotent of $S$,
    \item $\left[G/G_i\right]$ denotes a complete set of
      representatives of cosets modulo $G_i$,
    \item $\irr(G_i)$ denotes a complete set of representatives of
      the isomorphism classes of the irreducible representations of
      $G_i$,
    \item and ${\varepsilon}_U$ denotes a primitive idempotent of ${\mathbbm{k}} G_i$ such
      that $U \simeq {\mathbbm{k}} G_i\cdot {\varepsilon}_U$ in $\mathrm{mod}({\mathbbm{k}} G_i)$, for all
      $U\in \irr(G_i)$.
    \end{itemize}
  \item $\tilde e$ denotes the idempotent $\sum_{(i,U)} e_i
    *{\varepsilon}_U$ of $T_S(M)*G$, where $i$ runs through $[G\backslash Q_0]$
    and $U$ runs through $\irr(G_i)$.
  \end{itemize}
\end{notation}

\begin{ex}
  \label{sec:defin-main-results-8}
  Let $Q$ be the quiver
  \[
    \xymatrix@=5ex{
      & \bullet \ar@(lu,ru) \ar@<2pt>[rd] \ar@<2pt>[ld] \\
      \overline{0} \ar@(lu,l) \ar@<2pt>[ru] \ar@<2pt>[rr] &&
      \overline{1} \ar@(ru,r) \ar@<2pt>[lu] \ar@<2pt>[ll]
    }
  \]
  where $\overline{0}$ and $\overline{1}$ denote the classes in
  $\mathbb Z/2\mathbb Z$.
  For all vertices $i,j$ of $Q$, the arrow from $i$ to $j$ is denoted
  by $x_{i,j}$. Let $G$ be the group $\langle \tau\,|\,\tau^2\rangle$
  of order $2$. The actions of $G$ on $Q$ and on $M$ are assumed to be
  the ones such that
  \[
    \tau \cdot \bullet = \bullet,\ \ \tau\cdot \overline{0} = \overline{1}\,,
  \]
  and, for all vertices $i,j$ of $Q$,
  \[
    ^\tau x_{i,j} = -x_{\tau\cdot i,\tau \cdot j}\,.
  \]
  \begin{itemize}
  \item Let $[G\backslash Q_0]$ be $\{\bullet, \overline{0}\}$. Note
    that $G_{\bullet} = \{\mathrm{Id},\tau\}$ and
    $G_{\overline{0}} = \{\mathrm{Id}\}$. 
  \item Let $[G/G_\bullet]$ and
    $[G/G_{\overline{0}}]$ be $\{\mathrm{Id}\}$ and
    $\{\mathrm{Id},\tau\}$, respectively. 
  \item For all $k\in \mathbb Z/2\mathbb Z$, denote by ${\varepsilon}_k$ the
    idempotent $\frac{1}{2}(\mathrm{Id}+(-1)^k\tau)$; denote by
    $\rho_k$ the irreducible representation ${\mathbbm{k}} \cdot {\varepsilon}_k$ of $G$.
  \item Let $\mathrm{irr}(G_\bullet)$ and
    $\mathrm{irr}(G_{\overline{0}})$ be
    $\{\rho_k\ |\ k\in \mathbb Z/2\mathbb Z\}$ and
    $\{{\mathbbm{k}} \cdot \mathrm{Id}\}$, respectively.
    \end{itemize}
    In this situation,
    $\tilde e = e_\bullet * {\varepsilon}_{\overline{0}} + e_\bullet*{\varepsilon}_{\overline{1}} +
    e_{\overline{0}}*\mathrm{Id}$.
\end{ex}

\subsection{Demonet's quiver of the Morita reduction}
\label{sec:demon-quiv-morita}

Demonet proved in \cite{MR2578593} that the skew group algebra
$T_S(M)*G$ is Morita equivalent to the path algebra of a quiver as
follows.

\begin{set}
  \label{sec:defin-notat-main-10}
  Let $Q_G$ be a quiver with the following properties.
  \begin{itemize}
  \item The vertices are the couples $(i,U)$, where
    $i\in [G\backslash Q_0]$ and $U\in \irr(G_i)$.
  \item For all vertices $(i,U)$ and $(j,V)$, the arrows
    $(i,U) \to (j,V)$ form a basis over ${\mathbbm{k}}$ of
    $\Hom_{{\mathbbm{k}} G_{i}}(U, M(i,j;V))$, here
    $M(i,j;V)=\oplus_{y\in [G/G_j]}\,_iM_{y\cdot j}\otimes_{\mathbbm{k}} yV$, where
    $yV$ stands for $y \otimes_{\mathbbm{k}} V \subset {\mathbbm{k}} G\otimes_{\mathbbm{k}} V$, see
    Definition~\ref{sec:defin-notat-main-1} for details.
  \end{itemize}
\end{set}

\begin{ex}
  \label{sec:defin-main-results-9}
  This is the continuation of
  Example~\ref{sec:defin-main-results-8}. In order to describe a
  quiver $Q_G$ such as in Setting~\ref{sec:defin-notat-main-10},
  Table~\ref{tab:2} escribes the $G_i$-module $M(i,j;U)$ for all possible
  $i,j,U$.
  \begin{table}[!ht]
    \centering
    \begin{tabularx}{0.9\textwidth}{>{\hsize=0.2\hsize\raggedright}X>{\hsize=0.3\hsize\centering}X>{\hsize=0.5\hsize\centering\arraybackslash}X}
      \hline 
      $M(i,j;U)$ & Basis over ${\mathbbm{k}}$ & Action of $G_i$\\
      \hline
      $M(\bullet,\bullet;\rho_k)$
      &
        $x_{\bullet,\bullet}\otimes {\varepsilon}_k$
      &
        $^\tau(x_{\bullet,\bullet}\otimes {\varepsilon}_k) = (-1)^{k+1}
        x_{\bullet,\bullet} \otimes {\varepsilon}_k$
      \\
      $M(\bullet,\overline{0};{\mathbbm{k}}\cdot\mathrm{Id})$
             & $(x_{\bullet,k}\otimes \tau^k)_{k\in \mathbb Z/2\mathbb Z}$
                               &
                                 $^\tau(x_{\bullet,k}\otimes \tau^k) =
                                 -x_{\bullet,k+\overline{1}}\otimes
                                 \tau^{k+\overline{1}}$
      \\
      $M(\overline{0},\bullet;\rho_k)$
             & $x_{\overline{0},\bullet} \otimes{\varepsilon}_k$ 
                                   & Trivial
                                                   \\
      $M(\overline{0},\overline{0};{\mathbbm{k}}\cdot \mathrm{Id})$
                 & $(x_{\overline{0},k} \otimes \tau^k)_{k\in \mathbb Z/2\mathbb Z}$
                                   & Trivial
      \\
      \hline
    \end{tabularx}
    \caption{The $G_i$-modules $M(i,j;U)$ ($k\in \mathbb Z/2\mathbb
      Z$).}
    \label{tab:2}
  \end{table}

  For all $k\in \mathbb Z/2\mathbb Z$, denote by $a_k$ and $b_k$ the
  following respective morphisms of $G_{\bullet}$-modules
  \[
    \begin{array}{rclcrcl}
      \rho_{k+\overline{1}}
      & \longrightarrow
      &
        M(\bullet,\bullet;\rho_k)
      &&
         \rho_k
      & \longrightarrow
      &
        M(\bullet,\overline{0};{\mathbbm{k}} \cdot \mathrm{Id}) \\
      {\varepsilon}_{k + \overline{1}}
      &
        \longmapsto
      &
        x_{\bullet,\bullet}\otimes {\varepsilon}_k
      &&
         {\varepsilon}_k
      &
        \longmapsto
      &
        \frac{1}{2}(x_{\bullet,\overline{0}}\otimes
        \mathrm{Id} + (-1)^{k+\overline{1}}
        x_{\bullet,\overline{1}} \otimes \tau)
    \end{array}
  \]
  and denote by $c_k$ and $d_k$ the following respective morphisms of
  $G_{\overline{0}}$-modules
  \[
    \begin{array}{rclcrcl}
      {\mathbbm{k}} \cdot \mathrm{Id}
      & \longrightarrow
      &
        M(\overline{0},\bullet;\rho_k)
      &&
         {\mathbbm{k}} \cdot \mathrm{Id}
      & \longrightarrow
      &
        M(\overline{0},\overline{0};{\mathbbm{k}} \cdot
        \mathrm{Id}) \\
      \mathrm{Id}
      &
        \longmapsto
      &
        x_{\overline{0},\bullet}\otimes {\varepsilon}_k
      &&
      \mathrm{Id}
      &
        \longmapsto
      &
        (-1)^k x_{\overline{0},k}\otimes \tau^k\,.
    \end{array}
  \]
  Then, $Q_G$ may be taken equal to
  \[
    \xymatrix@R=2ex{
      (\bullet, \rho_{\overline{0}}) \ar@/_/[dd]_{a_{\overline{1}}}
      \ar@<2pt>[rrd]^{b_{\overline{0}}} & & \\
      & & (0,{\mathbbm{k}}\cdot \mathrm{Id}) \ar@<2pt>[llu]^{c_{\overline{0}}}
      \ar@<-2pt>[lld]_{c_{\overline 1}} \ar@(ul,ur)^{d_{\overline{0}}} \ar@(dr,dl)^{d_{\overline{1}}}\,,\\
      (\bullet, \rho_{\overline{1}}) \ar@/_/[uu]_{a_{\overline{0}}}
      \ar@<-2pt>[rru]_{b_{\overline{1}}} &&
      }
    \]
\end{ex}

Demonet's result may be reformulated as follows, see
Section~\ref{sec:proof-main-results} for more details.
\begin{thm}[{\cite{MR2578593}}]
  \label{sec:proof-main-results-1}
  The ${\mathbbm{k}}$-algebras ${\mathbbm{k}} Q_G$ and $T_S(M)*G$ are Morita
  equivalent. Besides, assuming that $U={\mathbbm{k}} G_i\cdot {\varepsilon}_U$ for all
  vertices $(i,U)$ of $Q_G$, then there exists an
  isomorphism of algebras
  \begin{equation}
    \label{eq:27}
    {\mathbbm{k}} Q_G \xrightarrow{\sim} \tilde e \cdot (T_S(M)*G) \cdot \tilde e
  \end{equation}
  which maps every arrow $f\colon (i,U) \to (j,V)$ of $Q_G$ to $f({\varepsilon}_U)$.
\end{thm}

\subsection{Monoidal categories as a toolbox for intertwiners}
\label{sec:tools-from-monoidal}

The present article introduces a ${\mathbbm{k}}$-algebra, called the
algebra of intertwiners relative to $M$ and denoted by $\intw$, and an
isomorphism of ${\mathbbm{k}}$-algebras
${\mathbbm{k}} Q_G\to \intw^{\mathrm{op}}$ together with explicit
formulas decomposing any given element of $\intw$ in the basis
consisting of the images of the paths in $Q_G$ under
${\mathbbm{k}} Q_G\to \intw^{\mathrm{op}}$. See
Theorem~\ref{sec:defin-main-results} for details. Composing this
isomorphism with the inverse of (\ref{eq:27}) results in a
${\mathbbm{k}}$-algebra isomorphism
$\tilde e \cdot (T_S(M)*G)\cdot \tilde e\to \intw^{\mathrm{op}}$. A
description of it is given in Theorem~\ref{sec:defin-main-results-4}
hence providing a decomposition, under (\ref{eq:27}), of any element
of $\tilde e \cdot (T_S(M)*G)\cdot \tilde e$ in the basis of
${\mathbbm{k}} Q_G$ consisting of the paths in $Q_G$. These
decompositions involve non-degenerate pairings between intertwiners
relative to $M$ and intertwiners relative to the dual vector space
$M^*$. Both the product in $\intw$ and the pairings are explicit
reformulations of natural operations and pairings in a suitable
monoidal category.

\begin{rem}[{\cite{MR3242743}}]
  \label{sec:defin-main-results-2}
  Let $({\mathcal C},\otimes)$ be a monoidal ${\mathbbm{k}}$-linear category with finite
  dimensional morphism spaces.
  \begin{enumerate}
  \item Given objects $X$, $Y$, $U$, $V$, and $W$ of ${\mathcal C}$, and
    following the ideas presented in \cite[Subsection 4.4]{G06} for the monoidal
    categories of finite subgroups of $\mathrm{SL}_3(\mathbb C)$,
    there is a natural operation between morphisms $U \to X\otimes V$
    and $V\to Y \otimes W$,
    \[
    \begin{array}{ccc}
      {\mathcal C}(U, X\otimes V) \times {\mathcal C}(V, Y \otimes W) & \longrightarrow
      & {\mathcal C}(U, X\otimes Y \otimes W) \\
      (f,f') & \longmapsto & (X\otimes f') \circ f\,.
    \end{array}
    \]
  \item Given an object $X$ of ${\mathcal C}$, recall that a \emph{left dual}
    of $X$ is an object $X'$ of ${\mathcal C}$ together with an adjunction
    $(X \otimes- )\vdash (X'\otimes-)$ of endofunctors of ${\mathcal C}$
    (\cite[Definition~2.10.1, Proposition~2.10.8]{MR3242743}). If
    $X$
    has a left dual, then there exists a non-degenerate pairing, for
    all objects $U$
    and $V$ of ${\mathcal C}$ such that ${\mathcal C}(U,U) = {\mathbbm{k}} \cdot \id_U$,
    \[
    \begin{array}{cccc}
      \langle - | - \rangle \colon & {\mathcal C}(U, X \otimes V) \otimes
                                     {\mathcal C}(V, X'\otimes U)
      &\longrightarrow & {\mathbbm{k}} \\
      & f \otimes \phi & \longmapsto & \langle f |\phi \rangle
    \end{array}
    \]
    such that
    $\langle f |\phi \rangle\cdot \id_U = \overline \phi \circ f$,
    where $\overline \phi \in {\mathcal C}(X \otimes V, U)$ is adjoint to
    $\phi$; because ${\mathcal C}(U,U)={\mathbbm{k}}\cdot \id_U$, a scalar is associated
    to every endomorphism of $U$, and hence to
    $\overline \phi \circ f$, which makes it possible to define a
    pairing with values in ${\mathbbm{k}}$. When ${\mathcal C}$ is a multitensor category
    this pairing is just a combination of the adjunction
    $(X \otimes- )\vdash (X'\otimes-)$ and the non-degenerate pairing
    ${\mathcal C}(U,X\otimes V) \otimes {\mathcal C}(X\otimes V,U)\to {\mathbbm{k}}$ which takes
    into account $U$ being simple in ${\mathcal C}$ (see
    \cite[Example~7.9.14]{MR3242743}).
  \end{enumerate}
\end{rem}

The monoidal category ${\mathcal C}$  considered here is the
category $\mathrm{mod}(A^e)$
of finite dimensional $A$-bimodules with tensor product being
$\otimes_A$, where $A$ is as follows,
\begin{equation}
  \label{eq:25}
  A = {\mathbbm{k}} \times \Pi_{i\in Q_0}{\mathbbm{k}} G_i\,.
\end{equation}
Note that $A$ is a semi-simple ${\mathbbm{k}}$-algebra; accordingly, any object
$X$ of ${\mathcal C}$ has a left dual given by the dual vector space $X^*$. The
isolated factor ${\mathbbm{k}}$ in the definition of $A$ is included so that
${\mathcal C}$ contains $\mathrm{mod}({\mathbbm{k}} G_i)$ as a full additive subcategory, yet not
as a monoidal subcategory, for all $i\in Q_0$. The reason for this
inclusion is the following, at the same time it clarifies the role of
$\mathcal C$. The text manipulates morphism spaces of the shape
\begin{equation}
  \label{eq:47}
  \Hom_{{\mathbbm{k}} G_{i_0}}(U,\,e_{i_0}(M \otimes_{\mathbbm{k}} {\mathbbm{k}} G)e_{i_1} \otimes_{{\mathbbm{k}}
  G_{i_1}} \cdots \otimes_{{\mathbbm{k}} G_{i_{n-1}}}e_{i_{n-1}}(M \otimes_{\mathbbm{k}} {\mathbbm{k}}
G)e_{i_n} \otimes_{{\mathbbm{k}} G_{i_n}} V)
\end{equation}
where $i_0,\ldots,i_n$ are vertices of $Q$ whereas
$U\in \mathrm{mod}({\mathbbm{k}} G_{i_0})$ and
$V\in \mathrm{mod}({\mathbbm{k}} G_{i_n})$; note that each one of
$e_{i_{t-1}}(M \otimes_{\mathbbm{k}} {\mathbbm{k}} G)e_{i_t}$ (for
$1\leqslant t\leqslant n$) is an object of $\mathcal C$, and so are
$U$ and $V$ because of the isolated factor ${\mathbbm{k}}$ in $A$; the
above-mentioned morphism space is hence equal to
\begin{equation}
  \label{eq:37}
  \mathcal C(U,\,e_{i_0}(M \otimes_{\mathbbm{k}} {\mathbbm{k}} G)e_{i_1} \otimes \cdots
  \otimes e_{i_{n-1}}(M \otimes_{\mathbbm{k}} {\mathbbm{k}} G)e_{i_n} \otimes
  V)
\end{equation}
and it is manipulated in this text using the tools recalled in
Remark~\ref{sec:defin-main-results-2}.%  Note that ${\mathcal C}$ is not the only
% monoidal category such that (\ref{eq:47}) can be described as
% (\ref{eq:37}); For instance, should there exist a couple $(j_0,j_1)$
% of vertices of $Q$ such that $G_{j_0} = G_{j_1}$, then removing the
% factor ${\mathbbm{k}} G_{j_1}$ form $A$ would yield a monoidal category such that
% (\ref{eq:47}) can still be described as (\ref{eq:37}); In view of
% implementation, the latter monoidal category might be considered as
% more efficient would be somehow simpler
% than ${\mathcal C}$. However, the reason for introducing ${\mathcal C}$ is to prove the
% needed properties of the product in $\intw$ and of the pairing;
% besides (\ref{eq:25}) has the advantage of being easy to describe.

\subsection{The interwiners}
\label{sec:interwiners}

Here are the base ingredients needed to define $\intw$.
\begin{defn}
  \label{sec:defin-notat-main-1}
  Let $n$ be a non-negative integer. Let
  $\underline i = i_0,\ldots,i_n$ be a sequence in $Q_0$. Let
  $V \in \mathrm{mod}({\mathbbm{k}} G_{i_n})$.
  \begin{enumerate}
  \item For all sequences $\underline y = y_1,\ldots,y_n$, where $y_t\in
    [G/G_{i_t}]$ for all $t\in \{1,\ldots,n\}$, let $M_{\underline
      y}(\underline i; V)$ be the following vector space,
    \[
    M_{\underline
      y}(\underline i; V) =
    \,_{i_0}M_{y_1\cdot i_1}\otimes_{\mathbbm{k}}\,_{y_1\cdot
    i_1}M_{y_1y_2\cdot i_2} \otimes_{\mathbbm{k}} \cdots \otimes_{\mathbbm{k}} \, _{y_1\cdots y_{n-1}\cdot
    i_{n-1}} M _{y_1\cdots y_n\cdot i_n} \otimes_{\mathbbm{k}} y_1\cdots y_n V\,,
  \]
  where $y_1\cdots y_n V$ stands for the vector subspace
  $y_1\cdots y_n \otimes_{\mathbbm{k}} V$ of ${\mathbbm{k}} G \otimes_{\mathbbm{k}} V$.
\item Define the ${\mathbbm{k}} G_{i_0}$-module $M(\underline i;V)$ as follows.
  Its underlying vector space is
      \begin{equation}
        \label{eq:42}
       \bigoplus_{\underline y} M_{\underline y}(\underline i;V)\,, 
      \end{equation}
      where $\underline y$ runs through all possible sequences such as
      in (1).  For all $g\in G_{i_0}$, the action of $g$ is written
      exponentially and defined as follows. For all $\underline y$,
      there exist a unique sequence $\underline y'=y_1',\ldots,y_n'$,
      where $y_t'\in [G/G_{i_t}]$ for all $t\in \{1,\ldots,n\}$, and a
      unique $h\in G_{i_n}$, such that
  \begin{itemize}
  \item $gy_1\cdots y_tG_{i_t} = y_1'\cdots y_t' G_{i_t}$ for all
    $t\in \{1,\ldots,n\}$,
  \item and $gy_1\cdots y_n = y_1'\cdots y_n'h$;
  \end{itemize}
  then, the action of $g$ transforms $M_{\underline y}(\underline
  i;V)$ into $M_{\underline y'}(\underline i;V)$ and, for all
  $m_1\otimes \cdots \otimes m_n \otimes y_1\cdots y_n v\in
  M_{\underline y}(\underline{i};V)$,
  \[
^g\left(m_1\otimes \cdots \otimes m_n\otimes y_1\cdots y_n v\right)
=
\,^gm_1\otimes \cdots \otimes \,^gm_n \otimes y_1' \cdots y_n'\,^hv\,.
  \]
\end{enumerate}
\end{defn}

This does define a finite dimensional ${\mathbbm{k}} G_{i_0}$-module. If $n=0$,
then $M(\underline i;V) = V$. In the monoidal category
$(\mathrm{mod}(A^e),\otimes_A)$ considered previously, $M(\underline i;V)$ is
isomorphic to $M(\underline i)\otimes_A V$ where $M(\underline i)$ is
a ${\mathbbm{k}} G_{i_0}-{\mathbbm{k}} G_{i_n}$-bimodule defined by $M$ and $\underline i$,
see (\ref{eq:24}) for the definition of $M(\underline i)$ and
(\ref{eq:subsp-mund-i-1}) for the isomorphism.

Note that, in the previous definition, if $U = {\mathbbm{k}} G_{i_0} \cdot {\varepsilon}_U$
and $V = {\mathbbm{k}} G_{i_n} \cdot {\varepsilon}_V$ for primitive idempotents
${\varepsilon}_U\in {\mathbbm{k}} G_{i_0}$ and ${\varepsilon}_V\in {\mathbbm{k}} G_{i_n}$, then $M(\underline
i;U)$ is a ${\mathbbm{k}} G_{i_0}$-submodule of $T_S(M)*G$.

\begin{ex}
  \label{sec:defin-main-results-10}
  This is the continuation of
  Example~\ref{sec:defin-main-results-9}. Let
  $k\in \mathbb Z/2\mathbb Z$. Then
  $M(\bullet, \overline{0},\overline{0},\bullet; \rho_k)$ identifies
  with the ${\mathbbm{k}} G_\bullet$-submodule of $T_S(M)*G$ whose underlying vector
  space is generated by
  $(x_{\bullet,a}x_{a,b}x_{b,\bullet}*{\varepsilon}_k)_{a,b\in \mathbb Z/2\mathbb
    Z}$. On this submodule, the action of $\tau$ is such that
  $^\tau(x_{\bullet,a}x_{a,b}x_{b,\bullet}*{\varepsilon}_k) =
  (-1)^{k+\overline{1}}
  x_{\bullet,a+\overline{1}}x_{a+\overline{1},b+\overline{1}}x_{b+\overline{1},\bullet}*{\varepsilon}_k$.
\end{ex}

The main results of this article are
based on a description of the quiver ${\mathbbm{k}} Q_G$ in terms of
intertwiners.

\begin{defn}
  \label{sec:defin-notat-main-4}
  An \emph{intertwiner} relative to $M$ is a morphism of
  ${\mathbbm{k}} G_{i_0}$-modules $U\to M(\underline i;V)$, for some sequence
  $\underline i = i_0,\ldots,i_n$ in $Q_0$ and some modules
  $U\in \mathrm{mod}({\mathbbm{k}} G_{i_0})$ and $V\in \mathrm{mod}({\mathbbm{k}} G_{i_n})$.
\end{defn}

Note that, in the previous definition, if $U = {\mathbbm{k}} G_{i_0} \cdot {\varepsilon}_U$
and $V = {\mathbbm{k}} G_{i_n} \cdot {\varepsilon}_V$ for primitive idempotents
${\varepsilon}_U\in {\mathbbm{k}} G_{i_0}$ and ${\varepsilon}_V\in {\mathbbm{k}} G_{i_n}$, then all intertwiners
$U\to M(\underline i;V)$ take their values in $T_S(M)*G$.

\begin{ex}
  \label{sec:defin-main-results-11}
  This is the continuation of
  Example~\ref{sec:defin-main-results-10}. Let $s,t\in \mathbb
  Z/2\mathbb Z$. For all $a,b\in \mathbb Z/2\mathbb Z$, the following
  holds true in $M(\bullet,\overline{0},\overline{0},\bullet;\rho_t)$,
  \[
    {\varepsilon}_s \cdot (x_{\bullet,a}x_{a,b}x_{b,\bullet} \otimes {\varepsilon}_t) =
    \frac{1}{2} ( x_{\bullet,a}x_{a,b}x_{b,\bullet} \otimes {\varepsilon}_t
    +(-1)^{s+\overline{1}+t}x_{\bullet,a+\overline{1}}x_{a+\overline{1},b+\overline{1}}x_{b+\overline{1},\bullet}
    \otimes {\varepsilon}_t)\,.
  \]
  Thus,
  $\Hom_{{\mathbbm{k}} G_\bullet}(\rho_s,M(\bullet,\overline{0},\overline{0},\bullet;\rho_t))$
  is two-dimensional with basis $(F_a)_{a\in \mathbb Z/2\mathbb Z}$,
  where $F_a$ denotes the following intertwiner relative to $M$, for
  all $a\in \mathbb Z/2\mathbb Z$,
  \[
    \begin{array}{rcl}
      \rho_s & \longrightarrow &
                                 M(\bullet,\overline{0},\overline{0},\bullet;\rho_t)
      \\
      {\varepsilon}_s & \longmapsto &
                           \frac{1}{2} ( x_{\bullet,\overline{0}}x_{\overline{0},a}x_{a,\bullet} \otimes {\varepsilon}_t
                           +(-1)^{s+\overline{1}+t}x_{\bullet,\overline{1}}x_{\overline{1},a+\overline{1}}x_{a+\overline{1},\bullet}
                           \otimes {\varepsilon}_t)\,.
    \end{array}
    \]
\end{ex}

The main purpose of this article is to provide explicit formulas in
the main theorems. These formulas use the following notation.

\begin{notation}
  \label{sec:defin-notat-main-5}
  Let $\underline i = i_0,\ldots,i_n$ be a sequence in $Q_0$ where
  $n\geqslant 1$. Let $U\in \mathrm{mod}({\mathbbm{k}} G_{i_0})$ and
  $V\in \mathrm{mod}({\mathbbm{k}} G_{i_n})$. Let $f\in \Hom_{{\mathbbm{k}} G_{i_0}}(U, M(\underline
  i; V))$. For all $u\in U$, write the decomposition of $f(u)$ along
  (\ref{eq:42}) as follows
  \[
  f(u) = \sum_{\underline y} f_{\underline y}(u)\,;
  \]
  and, for all sequences $\underline y = y_1,\ldots,y_n$, where
  $y_t \in [G/G_{i_t}]$ for all $t\in \{1,\ldots,n\}$, write
  symbolically the element $f_{\underline y}(u)$ of
  $M_{\underline y}(\underline i;V)$ as follows, with the sum sign
  omitted,
  \[
  f_{\underline{y}}(u) = f^{(1)}_{\underline{y}}(u) \otimes \cdots
  \otimes f^{(n)}_{\underline{y}}(u) \otimes y_1\cdots y_n
  f^{(0)}_{\underline{y}}(u)\,;
  \]
  hence, $f_{\underline y}^{(t)}(u)$ is meant to lie in
  $_{y_1\cdots y_{t-1}\cdot i_{t-1}} M_{y_1\cdots y_t \cdot i_t}$, for
  all $t\in \{1,\ldots,n\}$, and $f_{\underline y}^{(0)}(u)$ is meant
  to lie in $V$; as a whole,
  \begin{equation}
    \label{eq:11}
  f(u) = \sum_{\underline y} f^{(1)}_{\underline{y}}(u) \otimes \cdots
  \otimes f^{(n)}_{\underline{y}}(u) \otimes y_1\cdots y_n
  f^{(0)}_{\underline{y}}(u)\,.
\end{equation}
\end{notation}

\subsection{The algebra of intertwiners}
\label{sec:algebra-intertwiners}

Now here is the operation to be used as the multiplication in $\intw$.

\begin{defn}
  \label{sec:defin-notat-main-6}
  Let $\underline i'' = i_0,\ldots,i_{m+n}$ be a sequence in $Q_0$,
  where $m,n\geqslant 1$. Denote the sequences $i_0,\ldots,i_m$ and
  $i_m,i_{m+1},\ldots,i_{m+n}$ by $\underline i$ and $\underline i'$,
  respectively. Let $U\in \mathrm{mod}({\mathbbm{k}} G_{i_0})$, $V \in \mathrm{mod}({\mathbbm{k}} G_{i_m})$,
  and $W\in \mathrm{mod}({\mathbbm{k}} G_{i_{m+n}})$. For all $f\in \Hom_{{\mathbbm{k}}
    G_{i_0}}(U,M(\underline i; V))$ and $f'\in \Hom_{{\mathbbm{k}} G_{i_m}}(V,
  M(\underline i';W))$,
  define a ${\mathbbm{k}}$-linear mapping
\[
f' \circledast f \colon U \to M(\underline{i''} ; W)
\]
by
\begin{equation}
  \label{eq:36}
f' \circledast f = \sum_{\underline{y''}} (f'\circledast f)_{\underline{y''}}\,,
\end{equation}
where
\begin{itemize}
\item $\underline{y''}$ runs through all sequences
$y_1,\ldots,y_{m+n}$ with $y_t\in \left[ G/G_{i_t} \right]$ for all
$t$
\item  and $(f' \circledast f)_{\underline{y''}}$ is
the mapping $U \to M_{\underline y''}(\underline i'';W)$ such that,
for all $u\in U$,
\[
\begin{array}{rcl}
  (f' \circledast f)_{\underline{y''}}(u)
  & = & f^{(1)}_{\underline{y}}(u) \otimes \cdots \otimes
        f^{(m)}_{\underline{y}}(u) \otimes
  \\
  & & \,^{y_1\cdots y_m}(f_{\underline{y'}}'^{(1)}(f^{(0)}_{\underline{y}}(u))) \otimes \cdots \otimes
      \,^{y_1\cdots y_m}(f_{\underline{y'}}'^{(n)}(f^{(0)}_{\underline{y}}(u))) \otimes \\
  & & y_1\cdots y_{m+n} f_{\underline{y'}}'^{(0)}(f^{(0)}_{\underline{y}}(u)) \,,
\end{array}
\]
where $\underline{y} = y_1,\ldots,y_m$ and
$\underline{y'}=y_{m+1},\ldots,y_{m+n}$.
\end{itemize}
\end{defn}

Note that, in the previous definition, if $U = {\mathbbm{k}} G_{i_0} \cdot {\varepsilon}_U$,
$V = {\mathbbm{k}} G_{i_m} \cdot {\varepsilon}_V$, and $W = {\mathbbm{k}} G_{i_{m+n}}\cdot {\varepsilon}_W$ for
primitive idempotents ${\varepsilon}_U\in {\mathbbm{k}} G_{i_0}$, ${\varepsilon}_V\in {\mathbbm{k}} G_{i_m}$, and
${\varepsilon}_W \in {\mathbbm{k}} G_{i_{m+n}}$, then
\begin{equation}
  \label{eq:12}
  (f'\circledast f)({\varepsilon}_U) = f({\varepsilon}_U) \cdot f'({\varepsilon}_v)\,,
\end{equation}
where the product on the right-hand side is taken in $T_S(M)*G$.

\begin{rem}
  \label{sec:defin-main-results-5}
The definition of $\circledast$ is technical. Actually, it is an
explicit reformulation of a simple operation in the monoidal category
$(\mathrm{mod}(A^e),\otimes_A)$. More precisely, using the comment that
follows Definition~\ref{sec:defin-notat-main-1},
\begin{itemize}
\item first, $M(\underline i'') \simeq M(\underline i)\otimes_A
  M(\underline i')$, see Lemma~\ref{sec:spaces-munderline-i-3};
\item next, if $M(\underline i;V)$, $M(\underline i';W)$, and
  $M(\underline i'';W)$ are identified with $M(\underline i)\otimes_A
  V$, $M(\underline i')\otimes_A W$, and $M(\underline i'')\otimes_A
  W$, respectively, then $\circledast$ is an explicit reformulation of
  the operation introduced in part (1) of
  Remark~\ref{sec:defin-main-results-2}, see
  Lemma~\ref{sec:comp-intertw} for a precise statement and proof.
\end{itemize}
\end{rem}

In view of the previous remark, $\circledast$ does yield intertwiners,
see Lemma~\ref{sec:comp-intertw}, and it is associative, see
Proposition~\ref{sec:comp-intertw-2}.

\begin{defn}
  \label{sec:defin-main-results-3}
  The \emph{algebra of intertwiners} relative to $M$ is denoted by
  $\intw$ and defined by
  \begin{equation}
    \label{eq:26}
  \intw = \left(
    \bigoplus_{
    n,\underline i,U,V}
    \Hom_{{\mathbbm{k}} G_{i_0}}(U, M(\underline i;V)), \circledast
    \right)\,,
  \end{equation}
    where $n$ runs through all non negative integers, $\underline i =
    i_0,\ldots,i_n$ runs through all sequences in $[G\backslash Q_0]$,
    and $U$ and $V$ run through $\irr(G_{i_0})$ and $\irr(G_{i_n})$, respectively.
\end{defn}

It is hence possible to associate an element of $\intw$ to every path
in $Q_G$.
\begin{notation}
  \label{sec:defin-notat-main-11}
  For all paths in $Q_G$
  \[
   \gamma \colon (i_0,U_0)\xrightarrow{f_1} (i_1,U_1) \to\cdots \to (i_{n-1},U_{n-1})
  \xrightarrow{f_n} (i_n,U_n)\,,
  \]
  denote by $f_\gamma$ the following element of $\intw$ lying in $\Hom_{{\mathbbm{k}}
    G_{i_0}}(U_0,M(i_0,\ldots,i_n;U_n))$,
  \begin{equation}
    \label{eq:23}
    f_{\gamma}= f_n \circledast f_{n-1}\circledast \cdots \circledast f_1\,.
  \end{equation}
\end{notation}

\subsection{Dual intertwiners and their pairing with intertwiners}
\label{sec:dual-intertw-their}

As stated below, see Theorem~\ref{sec:defin-main-results}, assigning
$f_\gamma$ to $\gamma$ for every path $\gamma$ in $Q_G$ yields an
isomorphism of algebras from ${\mathbbm{k}} Q_G$ to $\intw^{\mathrm{op}}$. The purpose of
this article is to explain how to decompose any element of $\intw$ as
a linear combination of the $f_\gamma$'s. This involves non-degenerate
pairings between certain spaces of intertwiners relative to $M$ and to
$M^*$, respectively.  Here, $M^*$ denotes the dual vector space
$\Hom_{{\mathbbm{k}}}(M,{\mathbbm{k}})$, it has a structure of $S$-bimodule such that
$_iM^*_j = (\,_jM_i)^*$ for all $i,j\in Q_0$, and $G$ acts on $M^*$ in
such a way that, for all $\varphi\in M^*$ and $g\in G$,
\[
^g\varphi  = \varphi(\,^{g^{-1}}\bullet)\,.
\]
This action is such that $^g(\,_iM^*_j) =\,_{g\cdot i}M^*_{g\cdot j}$
for all $i,j\in Q_0$ and $g\in G$. Hence, the previous considerations
may be applied to $M^*$ instead of to $M$.

\begin{ex}
  \label{sec:defin-main-results-12}
  This is the continuation of
  Example~\ref{sec:defin-main-results-11}. For all vertices $i,j$ of
  $Q$, the vector space $_iM_j^*$ is one-dimensional; denote by
  $x_{i,j}'$ its base element, dual to the base element $x_{j,i}$ of
  $_jM_i$; Hence, $^{\tau}x_{i,j}' = -x_{\tau(i),\tau(j)}'$. After
  replacing $M$ by $M^*$, the considerations of
  Example~\ref{sec:defin-main-results-11} yield that
  $\Hom_{{\mathbbm{k}} G_\bullet}(\rho_t,M^*(\bullet,\overline{0},\overline{0},\bullet;\rho_s))$
  is two-dimensional with basis
  $(\phi_a)_{a\in \mathbb Z/2\mathbb Z}$, where $\phi_a$ denotes the
  following intertwiner relative to $M^*$, for all
  $a\in \mathbb Z/2\mathbb Z$,
  \[
    \begin{array}{rcl}
      \rho_t & \longrightarrow &
                                 M^*(\bullet,\overline{0},\overline{0},\bullet;\rho_s)
      \\
      {\varepsilon}_t & \longmapsto &
                           x_{\bullet,\overline{0}}x_{\overline{0},a}x_{a,\bullet} \otimes {\varepsilon}_s
                           +(-1)^{t+\overline{1}+s}x_{\bullet,\overline{1}}x_{\overline{1},a+\overline{1}}x_{a+\overline{1},\bullet}
                           \otimes {\varepsilon}_s)\,.
    \end{array}
  \]
\end{ex}

\begin{notation}
  \label{sec:defin-notat-main-12}
  Let $\underline i = i_0,\ldots,i_n$ be a sequence in $Q_0$, where
  $n\geqslant 0$. Let $U\in \mathrm{mod}({\mathbbm{k}} G_{i_0})$ and $V\in \mathrm{mod}({\mathbbm{k}}
    G_{i_n})$.
  \begin{enumerate}
  \item Denote by $\underline i^o$ the sequence
    $i_n,i_{n-1},\ldots,i_0$.
  \item Let $\varphi \in \Hom_{{\mathbbm{k}} G_{i_n}}(V, M^*(\underline
    i^o,U))$. Proceeding similarly as done in
    Notation~\ref{sec:defin-notat-main-5},
    \begin{itemize}
    \item denote by $\varphi^{\underline x}$ the composition of
      $\varphi$ with the canonical projection $M^*(\underline i^o;U)
      \to M^*_{\underline x}(\underline i^o;U)$, for all sequences
      $\underline x = x_{n-1},x_{n-2},\ldots,x_0$ where $x_t\in
      [G/G_{i_t}]$ for all $t\in \{0,\ldots,n-1\}$,
    \item and for all such $\underline x$ and all $v\in V$, write
      symbolically the element $\varphi^{\underline x}(v)$ of
      $M^*_{\underline x}(\underline i^o;U)$ as follows,
      \[
      \varphi^{\underline x}(v) =
      \varphi_{(n)}^{\underline{x}}(v) \otimes \cdots \otimes
      \varphi^{\underline{x}}_{(1)}(v) \otimes x_{n-1}\cdots x_0
      \varphi^{\underline{x}}_{(0)}(v)\,.
      \]
      Hence, $\varphi^{\underline x}_{(t)}(v)$ is meant to lie in
      $_{(x_{n-1}x_{n-2}\cdots x_t)\cdot i_t}M^*_{(x_{n-1}x_{n-2}\cdots
        x_{t-1})\cdot i_{t-1}}$ for all $t\in \{1,\ldots,n\}$ and
      $\varphi^{\underline x}_{(0)}(v)$ is meant to lie in $U$. As a
      whole,
      \begin{equation}
        \label{eq:17}
        \varphi(v) = \sum_{\underline x}\varphi^{\underline
          x}_{(n)}(v) \otimes \cdots \otimes \varphi^{\underline x}_{(1)}(v)
        \otimes x_{n-1}\cdots x_1 \varphi^{\underline x}_{(0)}(v)\,.
      \end{equation}
    \end{itemize}
  \end{enumerate}
\end{notation}

Here is an illustration of this notation.
\begin{ex}
  \label{sec:defin-main-results-16}
  This is the continuation of
  Example~\ref{sec:defin-main-results-12}. Assume that $n=2$ and
  $\underline i = \overline{0},\overline{0},\bullet$. Hence
  $\underline i^o = \bullet,\overline{0},\overline{0}$. The underlying
  vector space of the ${\mathbbm{k}} G_\bullet$-module $M^*(\underline i^o;{\mathbbm{k}}
  \cdot \mathrm{Id})$ is freely generated by $(x_{\bullet,a}'x_{a,b}'\otimes
  \tau^b)_{a,b\in \mathbb Z/2\mathbb Z}$.

  Let $a,b,s\in \mathbb Z/2\mathbb Z$. Denote by $\varphi$ the
  intertwiner $\rho_s \to M^*(\underline i^o;{\mathbbm{k}} \cdot \mathrm{Id})$
  such that
  $\varphi({\varepsilon}_s) = {\varepsilon}_s \cdot (x_{\bullet,a}'x_{a,b}'\otimes \tau^b)$,
  that is
  $\varphi({\varepsilon}_s) = \frac{1}{2}x_{\bullet,a}'x_{a,b}'\otimes \tau^b +
  \frac{(-1)^s}{2}x_{\bullet,a+\overline{1}}'x_{a+\overline{1},b+\overline{1}}'\otimes
  \tau^{b+\overline{1}}$. On one hand,
  $x_{\bullet,a}'x_{a,b}' = x_{\bullet,\tau^a\cdot
    \overline{0}}'x_{\tau^a\cdot \overline{0},\tau^a\tau^{b-a}\cdot
    \overline{0}}'$; Hence
  \[\tiny
    \left\{
      \begin{array}{rcl}
        \varphi^{\tau^a,\tau^{b-a}}({\varepsilon}_s)
        & = &
              \frac{1}{2}x_{\bullet,a}'x_{a,b}'\otimes \tau^b \\
        \varphi^{\tau^a,\tau^{b-a}}_{(2)}({\varepsilon}_s)
        \otimes
        \varphi^{\tau^a,\tau^{b-a}}_{(1)}({\varepsilon}_s)
        \otimes
        \varphi^{\tau^a,\tau^{b-a}}_{(0)}({\varepsilon}_s)
        & = &
              \frac{1}{2}x_{\bullet,a}'\otimes x_{a,b}'\otimes \mathrm{Id}\,.
      \end{array}\right.
    \]
    On the other hand,
    $x_{\bullet,a+\overline{1}}'x_{a+\overline{1},b+\overline{1}}' =
    x_{\bullet,\tau^{a+\overline{1}}\cdot \overline{0}}'x_{\tau^{a+\overline{1}}\cdot
      \overline{0},\tau^{a+\overline{1}}\tau^{b-a}\cdot \overline{0}}'$; Hence
  \[\tiny
    \left\{
      \begin{array}{rcl}
        \varphi^{\tau^{a+\overline{1}},\tau^{b-a}}({\varepsilon}_s)
        & = &
              \frac{(-1)^s}{2}x_{\bullet,a+\overline{1}}'x_{a+\overline{1},b+\overline{1}}'\otimes
              \tau^{b+\overline{1}}  \\
        \varphi^{\tau^{a+\overline{1}},\tau^{b-a}}_{(2)}({\varepsilon}_s)
        \otimes
        \varphi^{\tau^{a+\overline{1}},\tau^{b-a}}_{(1)}({\varepsilon}_s)
        \otimes
        \varphi^{\tau^{a+\overline{1}},\tau^{b-a}}_{(0)}({\varepsilon}_s)
        & = &
              \frac{(-1)^s}{2}x_{\bullet,a+\overline{1}}' \otimes x_{a+\overline{1},b+\overline{1}}'\otimes
              \mathrm{Id}\,.
      \end{array}\right.
  \]
  In view of description of $\varphi({\varepsilon}_s)$ made initially, all the
  remaining terms of the sum of (\ref{eq:17}) are zero.
\end{ex}

The above mentioned pairings are described explicitly as follows.

\begin{prop}
  \label{sec:defin-notat-main-13}
  Let $\underline i = i_0,\ldots,i_n$ be a sequence in $Q_0$, where
  $n\geqslant 0$. Let $U\in \mathrm{mod}({\mathbbm{k}} G_{i_0})$ and
  $V\in \mathrm{mod}({\mathbbm{k}} G_{i_n})$. Assume that $U$ is simple. There exists a
  non-degenerate pairing
\[
( -| - ) \colon \Hom_{{\mathbbm{k}} G_{i_0}}(U, M(\underline{i}; V)) \otimes_{\mathbbm{k}}
\Hom_{{\mathbbm{k}} G_{i_n}}(V, M^*(\underline{i}^o; U)) \to {\mathbbm{k}} 
\]
with the following property. For all
$f\in \Hom_{{\mathbbm{k}} G_{i_0}}(U, M(\underline{i}; V))$ and for all
$\varphi \in \Hom_{{\mathbbm{k}} G_{i_n}}(V, M^*(\underline{i}^o; U))$, the
scalar $(f|\varphi)$ is such that, for all $u\in U$, the element
$(f|\varphi)\,u$ of $U$ is equal to
  \begin{equation}
    \label{eq:29}
    \sum_{\underline{y}} \prod_{t=1}^n
    \varphi^{\underline{x}}_{(t)}(f_{\underline{y}}^{(0)}(u)) (
    ^{(y_1\cdots y_n)^{-1}} f_{\underline{y}}^{(t)}(u)) 
    \ ^{h_0}(\varphi^{\underline{x}}_{(0)}(f_{\underline{y}}^{(0)}(u)))\,,
  \end{equation}
  where $\underline{y}=y_1,\ldots,y_n$ runs through all sequences with
  $y_t\in [G/G_{i_t}]$ for all $t$ and
  $\underline{x} = x_{n-1},\ldots,x_0$ and $h_0$ are uniquely
  determined by the following conditions:
  \begin{enumerate}[(a)]
  \item $x_t\in \left[G/G_{i_t}\right]$ for all $t$, and $h_0\in
    G_{i_0}$.
  \item $(y_t\cdots y_n)^{-1} \in  x_{n-1}x_{n-2}\cdots x_{t-1}G_{i_{t-1}}$ for
    all $t\in \{1,\ldots,n\}$.
  \item $(y_1\cdots y_n)^{-1} = x_{n-1}x_{n-2}\cdots x_0h_0$.
  \end{enumerate}
\end{prop}

In view of computations, here is a useful explanation regarding
$\underline x$ and $h_0$ in the above proposition. For all sequences
$\underline y = y_1,\ldots, y_n$ in $G$ such that $y_t\in [G/G_{i_t}]$
for all $t\in \{1,\ldots,n\}$ and for all non zero tensors
\[
m_1\otimes \cdots \otimes m_n \in \,_{i_0}M_{y_1\cdot i_1} \otimes_{\mathbbm{k}}
\cdots \otimes_{\mathbbm{k}} \, _{(y_1\cdots y_{n-1})\cdot i_{n-1}} M _{(y_1\cdots
  y_n) \cdot i_n}\,,
\]
there exists a unique sequence $\underline x =
x_{n-1},x_{n-2},\ldots,x_0$ in $G$ such that $x_t\in [G/G_{i_t}]$ for
all $t\in \{0,\ldots,n-1\}$ and
\[
^{(y_1\cdots y_n)^{-1}}(m_1 \otimes \cdots \otimes m_n) \in
\,_{(x_{n-1} x_{n-2}\cdots x_0)\cdot i_0}M_{(x_{n-1} \cdots x_1)\cdot i_1}
\otimes_{\mathbbm{k}} \cdots \otimes_k \, _{x_{n-1}\cdot i_{n-1}} M _{i_n}\,;
\]
besides,
$(x_{n-1}x_{n-2}\cdots x_0)^{-1}(y_1\cdots y_n)^{-1} \in G_{i_0}$.
This sequence $\underline x$ is the one in the statement of
Proposition~\ref{sec:defin-notat-main-13} in which
$h_0 = (x_{n-1}x_{n-2}\cdots x_0)^{-1}(y_1\cdots y_n)^{-1}$.

When the action of $G$ on $M$ is such that arrows of $Q$ are
transformed into scalar multiples of arrows of $Q$, the
expression~(\ref{eq:29}) of $(-|-)$ simplifies to (\ref{eq:51}). The
following example illustrates the computation of $(-|-)$ in this
particular situation.
\begin{ex}
  \label{sec:defin-main-results-13}
  This is the continuation of
  Example~\ref{sec:defin-main-results-12}. Let
  $s,t\in \mathbb Z/2\mathbb Z$. Let $a,b\in \mathbb Z/2\mathbb Z$.
  Consider
  $F_a \colon \rho_s \to
  M(\bullet,\overline{0},\overline{0},\bullet;\rho_t)$ and
  $\phi_b \colon \rho_t \to
  M(\bullet,\overline{0},\overline{0},\bullet;\rho_s)$ such as in
  Examples~\ref{sec:defin-main-results-11} and
  \ref{sec:defin-main-results-12}, respectively.

  In view of (\ref{eq:51}), $(F_a|\phi_b)$ is a sum of products of
  scalars, the sum is indexed by the paths in $Q$ that appear in the
  expression of $F_a({\varepsilon}_s)$. For all such paths $\gamma$, the scalars
  forming the corresponding product appear in Table~\ref{tab:3} with
  bold face. For convenience, this table presents additional data,
  using the following notation, see
  Subsection~\ref{sec:how-pair-intertw} for more details. For a given
  path $\gamma$ in $Q$:
  \begin{itemize}
  \item $\alpha_\gamma$ denotes the coefficient of $\gamma\otimes{\varepsilon}_t$
    in $F_a({\varepsilon}_s)$;
  \item $y_1,y_2,y_3$ is the sequence of $G$ such that $y_1,y_2\in
    [G/G_{\overline{0}}]$, $y_3\in [G/G_\bullet]$ and $\gamma$ is a
    path of the shape $\bullet \to y_1\cdot \overline{0}\to
    y_1y_2\cdot\overline{0}\to y_1y_2y_3\cdot \bullet$;
  \item $(y_1y_2y_3)^{-1}(\gamma)$ denotes the path of $Q$ and
    $\chi_{(y_1y_2y_3)^{-1},\gamma}$ denotes the scalar such that
    $^{(y_1y_2y_3)^{-1}}\gamma = \chi_{(y_1y_2y_3)^{-1},\gamma} \cdot
    (y_1y_2y_3)^{-1}(\gamma)$;
  \item $(y_1y_2y_3)^{-1}(\gamma)^*$ denotes the element of the basis
    of $T_S(M^*)$ dual to the basis of $T_S(M)$ consisting of the
    paths of $Q$, associated to $(y_1y_2y_3)^{-1}(\gamma)$;
  \item $\beta_{(y_1y_2y_3)^{-1}(\gamma)^*}$ denotes the coefficient
    of $(y_1y_2y_3)^{-1}(\gamma)^*$ in $\phi_b({\varepsilon}_t)$;
  \item $\chi_{\rho_s}$ denotes the character of $\rho_s$.
  \end{itemize}
  Thus,
  \begin{equation}
    \label{eq:56}
    (F_a|\phi_b) = \left\{
      \begin{array}{ll}
        \frac{(-1)^{(s+\overline{1})a}\cdot (1+(-1)^t)}{2} & \text{if $a=b$}
        \\
        0 & \text{if not.}
      \end{array}\right.
  \end{equation}
    
    \begin{table}[!h]
    \centering
    \begin{tabularx}{\textwidth}{>{\hsize=0.25\hsize\raggedright}X>{\hsize=0.35\hsize\centering}X>{\hsize=0.4\hsize\centering\arraybackslash}X}
      \hline
      $\gamma$ &
               $x_{\bullet,\overline{0}}x_{\overline{0},a}x_{a,\bullet}$
      &
       $x_{\bullet,\overline{1}}x_{\overline{1},a+\overline{1}}x_{a+\overline{1},\bullet}$
      \\
      \hline
      $\mathbf{\alpha_\gamma}$
               &
                 $\mathbf{\frac{1}{2}}$
      &
        $\mathbf{\frac{(-1)^{s+\overline{1}+t}}{2}}$\\
      $y_1,y_2,y_3$ & $\mathrm{Id},\tau^a,\mathrm{Id}$ &
                                                         $\tau,\tau^a,\mathrm{Id}$
      \\
      $^{(y_1y_2y_3)^{-1}}\gamma$
               &
                 $(-1)^ax_{\bullet,-a}x_{-a,\overline{0}}x_{\overline{0},\bullet}$
      &
        $(-1)^{a+\overline{1}}x_{\bullet,-a}x_{-a,\overline{0}}x_{\overline{0},\bullet}$
      \\
      $\mathbf{\chi_{(y_1y_2y_3)^{-1},\gamma}}$
               &
                 $\mathbf{(-1)^a}$
      &
        $\mathbf{(-1)^{a+\overline{1}}}$
      \\
      $(y_1y_2y_3)^{-1}(\gamma)^*$
               &
                 $x_{\bullet,\overline{0}}'x_{\overline{0},-a}'x_{-a,\bullet}'$
      &
        $x_{\bullet,\overline{0}}'x_{\overline{0},-a}'x_{-a,\bullet}'$
      \\
      $\mathbf{\beta_{(y_1y_2y_3)^{-1}(\gamma)^*}}$
               &
                 $\mathbf{\left\{
                 \begin{array}{ll}
                   1 & \text{if $b=-a$}\\
                   0 & \text{if not}
                 \end{array}\right.}$
      &
                 $\mathbf{\left\{
                 \begin{array}{ll}
                   1 & \text{if $b=-a$}\\
                   0 & \text{if not}
                 \end{array}\right.}$
      \\
      $x_2,x_1,x_0,h_0$
               &
                 $\mathrm{Id},\tau^{-a},\mathrm{Id},\tau^{-a}$
      &
        $\mathrm{Id},\tau^{-a},\mathrm{Id},\tau^{-a-\overline{1}}$
      \\
      $\mathbf{\chi_{\rho_s}(h_0)}$
               &
                 $\mathbf{(-1)^{sa}}$
      &
        $\mathbf{(-1)^{s(a+\overline{1})}}$
      \\
      {\tiny contribution of $\gamma$ to $(F_a|\phi_b)$}
               &
                 $\mathbf{\left\{
                 \begin{array}{ll}
                   \frac{(-1)^{a(s+\overline{1})}}{2} & \text{if $b=-a$}\\
                   0 & \text{if not}
                 \end{array}\right.}$
      &        
                 $\mathbf{\left\{
                 \begin{array}{ll}
                   \frac{(-1)^{(s+\overline{1})a+t}}{2} & \text{if $b=-a$}\\
                   0 & \text{if not}
                 \end{array}\right.}$
      \\
      \hline
    \end{tabularx}
    \caption[$(F_a|\phi_b)$]{The scalars appearing in the expression
      (\ref{eq:51}) of $(F_a|\phi_b)$}
    \label{tab:3}
  \end{table}
\end{ex}
\begin{rem}
  \label{sec:defin-main-results-6}
Like the operation $\circledast$, the pairing $(-|-)$ has a simple
interpretation in the monoidal category $(\mathrm{mod}(A^e),\otimes_A)$. Indeed,
\begin{itemize}
\item first, $M(\underline i)$ admits $M^*(\underline i^o)$ as a left
  dual, see~(\ref{eq:6});
\item next, if $M(\underline i;V)$ and $M^*(\underline i^o;U)$ are
  identified with $M(\underline i) \otimes_A V$ and $M^*(\underline
  i^o)\otimes_A U$, respectively, then $(-|-)$ is an explicit
  reformulation of $\langle -|-\rangle$ as introduced in part (2) of
  Remark~\ref{sec:defin-main-results-2}, see~(\ref{eq:20}).
\end{itemize}
\end{rem}

Using $(-|-)$ yields intertwiners relative to $M^*$ associated
with the $f_{\gamma}$'s for paths $\gamma$ in $Q_G$.

\begin{notation}
  \label{sec:defin-main-results-1}
  Let $Q_G$ be such as in Setting~\ref{sec:defin-notat-main-10}.
  \begin{enumerate}
  \item Let $f\in (i,U) \to (j,V)$ be an arrow of $Q_G$. Using the
    basis of the vector space $\Hom_{{\mathbbm{k}} G_i}(U, M(i,j;V))$ consisting
    of the arrows $(i,U) \to (j,V)$ of $Q_G$ and using the associated
    $(-|-)$-dual basis of $\Hom_{{\mathbbm{k}} G_j}(V, M^*(j,i;U))$, denote by
    $f^\vee$ the dual element associated with $f$. In other words,
    $f^\vee$ lies in $\Hom_{{\mathbbm{k}} G_j}(V, M^*(j,i;U))$ and, for all arrows
    $f'\colon (i,U) \to (j,V)$ of $Q_G$,
    \[
    (f' |f^\vee) =
    \left\{
      \begin{array}{ll}
        1 & \text{if $f'=f$} \\
        0 & \text{otherwise.}
      \end{array}\right.
    \]
  \item For all paths $\gamma$ in $Q_G$
    \[
    \gamma \colon (i_0,U_0)\xrightarrow{f_1} (i_1,U_1) \to\cdots \to
    (i_{n-1},U_{n-1}) \xrightarrow{f_n} (i_n,U_n)\,,
    \]
    denote by $\varphi_\gamma$ the following intertwiner
    \[
    \varphi_\gamma = f_1^\vee \circledast f_2^\vee \circledast \cdots
    \circledast f_n^\vee \in \Hom_{{\mathbbm{k}}
      G_{i_n}}( U_n, M^*(i_n,i_{n-1},\ldots,i_0;U_0))\,.
    \]
  \end{enumerate}
\end{notation}

\begin{ex}
  \label{sec:defin-main-results-14}
  This is the continuation of Example~\ref{sec:defin-main-results-13},
  it illustrates Notation~\ref{sec:defin-main-results-1}.

  Let $k\in \mathbb Z/2\mathbb Z$. Denote by $a_k'$ and $b_k'$ the
  following respective intertwiners relative to $M^*$
  \[
    \begin{array}{rclcrcl}
      \rho_k
      & \longrightarrow
      &
        M^*(\bullet,\bullet;\rho_{k+\overline{1}})
      &&
         {\mathbbm{k}} \cdot \mathrm{Id}
      & \longrightarrow
      &
        M^*(\overline{0},\bullet;\rho_k) \\
      {\varepsilon}_k
      &
        \longmapsto
      &
        x_{\bullet,\bullet}'\otimes {\varepsilon}_{k + \overline{1}}
      &&
         \mathrm{Id}
      &
        \longmapsto
      &
        x_{\overline{0},\bullet}'\otimes {\varepsilon}_k\,.
    \end{array}
  \]
  Denote by $c_k'$ and $d_k'$ the following respective
  intertwiners relative to $M^*$
  \[
    \begin{array}{rclcrcl}
      \rho_k
      & \longrightarrow
      &
        M^*(\bullet,\overline{0};{\mathbbm{k}} \cdot \mathrm{Id})
      &&
         {\mathbbm{k}} \cdot \mathrm{Id}
      & \longrightarrow
      &
        M^*(\overline{0},\overline{0};{\mathbbm{k}} \cdot
        \mathrm{Id}) \\
      {\varepsilon}_k
      &
        \longmapsto
      &
        (x_{\bullet,\overline{0}}' \otimes \mathrm{Id} +
        (-1)^{k+\overline{1}} x_{\bullet,\overline{1}}' \otimes \tau)
      &&
         \mathrm{Id}
      &
        \longmapsto
      &
        x_{\overline{0},k}' \otimes \tau^k\,.
    \end{array}
  \]
  Then $(a_k|a_k')=1$, $(b_k|b_k')=1$, $(c_k|c_k')=1$ and
  $(d_k|d_\ell')=\delta_{k,\ell}$ for all
  $\ell\in \mathbb Z/2\mathbb Z$, where $\delta_{k,\ell}$ stands for
  the Kronecker symbol. Using~(\ref{eq:51}), the intermediate steps of
  the computation of these pairings are given in Tables~\ref{tab:4},
  \ref{tab:5}, \ref{tab:6} and \ref{tab:7} (see
  Example~\ref{sec:defin-main-results-13}). In view of the shape of
  $Q_G$, these equalilities entail that $a_k^\vee = a_k'$,
  $b_k^\vee=b_k'$, $c_k^\vee=c_k$ and $d_k^\vee=d_k'$.
  \begin{table}[!ht]
    \tiny
    \centering
    \begin{tabular}{cccccccccc}
      \hline
      $\gamma$
      & $\mathbf{\alpha_\gamma}$
      &
        $y_1$
      &
        $^{y_1^{-1}}\gamma$
      &
        $\mathbf{\chi_{y_1^{-1},\gamma}}$
      &
        $y_1^{-1}(\gamma)^*$
      &
        $\mathbf{\beta_{y_1^{-1}(\gamma)^*}}$
      &
        $x_0,h_0$
      &
        $\mathbf{\chi_{\rho_{k+\overline{1}}}(h_0)}$
      &
        {\tiny contr. of $\gamma$}
      \\
      \hline
      $x_{\bullet,\bullet}$
      &
        $\mathbf{\frac{1}{2}}$
      &
        $\mathrm{Id}$
      &
        $x_{\bullet,\bullet}$
      &
        $\mathbf{1}$
      &
        $x_{\bullet,\bullet}'$
      &
        $\mathbf{1}$
      &
        $\mathrm{Id},\mathrm{Id}$
      &
        $\mathbf{1}$
      &
        $\mathbf{1}$
      \\
      \hline
    \end{tabular}
    \caption{$(a_k|a_k')=1$}
    \label{tab:4}
  \end{table}
  \begin{table}[!ht]
    \tiny
    \centering
    \begin{tabular}{cccccccccc}
      \hline
      $\gamma$
      & $\mathbf{\alpha_\gamma}$
      &
        $y_1$
      &
        $^{y_1^{-1}}\gamma$
      &
        $\mathbf{\chi_{y_1^{-1},\gamma}}$
      &
        $y_1^{-1}(\gamma)^*$
      &
        $\mathbf{\beta_{y_1^{-1}(\gamma)^*}}$
      &
        $x_0,h_0$
      &
        $\mathbf{\chi_{\rho_k}(h_0)}$
      &
        {\tiny contr. of $\gamma$}
      \\
      \hline
      $x_{\bullet,\overline{0}}$
      &
      $\mathbf{\frac{1}{2}}$
      &
      $\mathrm{Id}$
      &
      $x_{\bullet,\overline{0}}$
      &
      $\mathbf{1}$
      &
      $x_{\overline{0},\bullet}'$
      &
      $\mathbf{1}$
      &
      $\mathrm{Id},\mathrm{Id}$
      &
      $\mathbf{1}$
      &
      $\mathbf{\frac{1}{2}}$
      \\
      $x_{\bullet,\overline{1}}$
      &
        $\mathbf{\frac{(-1)^{k+\overline{1}}}{2}}$
      &
        $\tau$
      &
        $-x_{\bullet,\overline{0}}$
      &
        $\mathbf{-1}$
      &
        $x_{\overline{0},\bullet}'$
      &
        $\mathbf{1}$
      &
        $\mathrm{Id},\tau$
      &
        $\mathbf{(-1)^k}$
      &
        $\mathbf{\frac{1}{2}}$
      \\
      \hline
    \end{tabular}
    \caption{$(b_k|b_k')=\frac{1}{2}+\frac{1}{2}=1$}
    \label{tab:5}
  \end{table}
  \begin{table}[!ht]
    \tiny
    \centering
    \begin{tabular}{cccccccccc}
      \hline
      $\gamma$
      & $\mathbf{\alpha_\gamma}$
      &
        $y_1$
      &
        $^{y_1^{-1}}\gamma$
      &
        $\mathbf{\chi_{y_1^{-1},\gamma}}$
      &
        $y_1^{-1}(\gamma)^*$
      &
        $\mathbf{\beta_{y_1^{-1}(\gamma)^*}}$
      &
        $x_0,h_0$
      &
        $\mathbf{\chi_{\rho_k}(h_0)}$
      &
        {\tiny contr. of $\gamma$}
      \\
      \hline
      $x_{\overline{0},\bullet}$
      &
      $\mathbf{1}$
      &
      $\mathrm{Id}$
      &
      $x_{\overline{0},\bullet}$
      &
      $\mathbf{1}$
      &
      $x_{\bullet,\overline{0}}'$
      &
      $\mathbf{1}$
      &
      $\mathrm{Id},\mathrm{Id}$
      &
      $\mathbf{1}$
      &
      $\mathbf{\frac{1}{2}}$
      \\
      \hline
    \end{tabular}
    \caption{$(c_k|c_k')=1$}
    \label{tab:6}
  \end{table}
  \begin{table}[!ht]
    \tiny
    \centering
    \begin{tabular}{cccccccccc}
      \hline
      $\gamma$
      & $\mathbf{\alpha_\gamma}$
      &
        $y_1$
      &
        $^{y_1^{-1}}\gamma$
      &
        $\mathbf{\chi_{y_1^{-1},\gamma}}$
      &
        $y_1^{-1}(\gamma)^*$
      &
        $\mathbf{\beta_{y_1^{-1}(\gamma)^*}}$
      &
        $x_0,h_0$
      &
        $\mathbf{\chi_{{\mathbbm{k}}\cdot\mathrm{Id}}(h_0)}$
      &
        {\tiny contr. of $\gamma$}
      \\
      \hline
      $x_{\overline{0},k}$
      &
      $\mathbf{(-1)^k}$
      &
      $\mathrm{\tau^k}$
      &
      $(-1)^kx_{k,\overline{0}}$
      &
      $\mathbf{(-1)^k}$
      &
      $x_{\overline{0},k}'$
      &
        $\mathbf{\delta_{k,\ell}}$
      &
      $\tau^{-k},\mathrm{Id}$
      &
        $\mathbf{1}$
      &
        $\mathbf{\delta_{k,\ell}}$
      \\
      \hline
    \end{tabular}
    \caption{$(d_k|d_\ell')=\delta_{k,\ell}$}
    \label{tab:7}
  \end{table}

  Now it is possible to give examples of intertwiners relative to
  $M^*$ and associated to paths in $Q_G$. Let
  $s,t,a\in \mathbb Z/2\mathbb Z$. Denote by $\gamma$ the following
  path of $Q_G$.
  \[
    (\bullet,\rho_s) \xrightarrow{b_s} (\overline{0},{\mathbbm{k}}\cdot
    \mathrm{Id}) \xrightarrow{d_a} (\overline{0},{\mathbbm{k}}\cdot
    \mathrm{Id}) \xrightarrow{c_t} (\bullet,\rho_t)\,.
  \]
  Then $\varphi_\gamma = b_s^\vee \circledast d_a^\vee \circledast
  c_t^\vee = b_s' \circledast d_a' \circledast
  c_t'$. This is the intertwiner $\rho_t\to
  M^*(\bullet,\overline{0},\overline{0},\bullet;\rho_s)$ such that
  (see~(\ref{eq:12}))
  \[
    \begin{array}{rcl}
      \varphi_\gamma({\varepsilon}_t)
      & = &
            c_t'({\varepsilon}_t) \cdot d_a'(\mathrm{Id}) \cdot b_s'(\mathrm{Id})
      \\
      & = &
            (x_{\bullet,\overline{0}}'\otimes \mathrm{Id} +
            (-1)^{t+\overline{1}}x_{\bullet,\overline{1}}' \otimes
            \tau)
            \cdot
            (x_{\overline{0},a}' \otimes \tau^a)
            \cdot
            (x_{\overline{0},\bullet}' \otimes {\varepsilon}_s)
      \\
      & = &
            (-1)^{a(s+\overline{1})}
            (x_{\bullet,\overline{0}}'x_{\overline{0},a}'x_{a,\bullet}' \otimes {\varepsilon}_s +
            (-1)^{t+\overline{1}+s}
            x_{\bullet,\overline{1}}'x_{\overline{1},a+\overline{1}}'x_{a+\overline{1},\bullet}' \otimes {\varepsilon}_s)\,.
    \end{array}
    \]
    Thus, using the notation introduced in
    Example~\ref{sec:defin-main-results-12}, $\varphi_\gamma = (-1)^{a(s+\overline{1})}\phi_a$.
\end{ex}

\subsection{The comparison of the algebra of intertwiners with the
  Morita reduction}
\label{sec:comp-algebra-intertw}

Now, here is the result comparing ${\mathbbm{k}} Q_G$ and $\intw$.

\begin{thm}
  \label{sec:defin-main-results}
  Keep Settings~\ref{sec:defin-notat-main-2} and
  \ref{sec:defin-notat-main-10}.
  \begin{enumerate}
  \item Assigning $f_\gamma$ to every path
  $\gamma$ in $Q_G$ yields an isomorphism of algebras
  \begin{equation}
    \label{eq:32}
    {\mathbbm{k}} Q_G \xrightarrow{\sim}\intw^{\mathrm{op}}\,.
  \end{equation}
\item For all $f\in \intw$, say lying in
  $\Hom_{{\mathbbm{k}} G_{i_0}}(U_0, M(i_0,\ldots,i_n; U_n))$,
  \begin{equation}
    \label{eq:28}
  f = \sum_{\gamma} (f|\varphi_\gamma)\,f_\gamma\,,
\end{equation}
where $\gamma$ runs through all paths of $Q_G$ of the shape
$(i_0,U_0)\to (i_1,\bullet)\to \cdots \to (i_{n-1},\bullet) \to
(i_n,U_n)$.
  \end{enumerate}
\end{thm}

It is now possible to explain how to decompose any element of
$\tilde e \cdot (T_S(M)*G)\cdot \tilde e$ along the paths in $Q_G$
\emph{via} the isomorphism
${\mathbbm{k}} Q_G \xrightarrow{~(\ref{eq:27})} \tilde e (T_S(M)*G)\cdot \tilde
e$.
This is done with the following construction and assuming that
$U= {\mathbbm{k}} G_i \cdot {\varepsilon}_U$ for all vertices $(i,U)$ of $Q_G$. Given a
tensor
$m_1\otimes \cdots \otimes m_n *g \in e_{i_0}\cdot
(M^{\otimes_Sn}*G)\cdot e_{i_n}$,
where $n\geqslant 0$ and $i_0,i_n\in [G\backslash Q_0]$, then
\[
m_1\otimes \cdots \otimes m_n *g \in \left( \,_{i_0}M_{y_1\cdot i_1}
  \otimes_{{\mathbbm{k}}} \cdots \otimes_{{\mathbbm{k}}} \,_{(y_1\cdots y_{n-1})\cdot i_{n-1}}
  M _{(y_1\cdots y_{n})\cdot i_{n}} \right) * y_1\cdots y_n{\mathbbm{k}} G_{i_n}
\]
for unique $i_1,\ldots,i_{n-1}\in [G\backslash Q_0]$ and
$y_t\in [G/G_{i_t}]$, for all $t\in \{1,\ldots,n\}$; this defines an
intertwiner, for all $U\in \irr(G_{i_0})$ and $V\in \irr(G_{i_n})$,
\[
\begin{array}{ccc}
  {\mathbbm{k}} G_{i_0}\cdot {\varepsilon}_U & \longrightarrow
  &
    M(i_0,\ldots,i_n;{\mathbbm{k}} G_{i_n}\cdot {\varepsilon}_V) \\
  u \cdot {\varepsilon}_U & \longmapsto &
    (u*{\varepsilon}_U) \cdot (m_1\otimes\cdots\otimes m_n *g) \cdot
    (e_{i_n} * {\varepsilon}_V)\,,
\end{array}
\]
where the products on the right-hand side are meant in $T_S(M)*G$;
summing these intertwiners over all $U\in \irr(G_{i_0})$ and
$V\in \irr(G_{i_n})$ defines an element of $\intw$. This construction
may be extended linearly to a linear mapping
\[
\hat{e} \cdot (T_S(M)*G)\cdot \hat{e} \longrightarrow \intw^{\mathrm{op}}\,,
\]
where $\hat{e} = \sum_{i\in [G\backslash Q_0]}e_i$.  Denote by
  $\Xi$ the restriction of this mapping to
  $\tilde e\cdot (T_S(M)*G)\cdot \tilde e$,
\begin{equation}
  \label{eq:43}
  \Xi \colon \tilde e\cdot (T_S(M)*G) \cdot \tilde e \longrightarrow \intw^{\mathrm{op}}\,.
\end{equation}

\begin{rem}
  \label{sec:defin-main-results-7}
  Here is how $\hat{e}$ and $\tilde{e}$ are related. Given idempotents
  $a$ and $b$ of an algebra $\Lambda$, say that $b$ is a
  \emph{reduction} of $a$ whenever $\Lambda b$ is a direct summand of
  $\Lambda a$ in $\mathrm{Mod}(\Lambda)$ and the algebras
  $\mathrm{End}_{\Lambda}(\Lambda b)$ and
  $\mathrm{End}_{\Lambda}(\Lambda a)$ are Morita equivalent. Now
  $\tilde e$
  ($=\sum_{i\in [G\backslash Q_0]} e_i *(\sum_{U\in \irr(G_i)}{\varepsilon}_U)$)
  is a reduction of $\hat e * 1_{{\mathbbm{k}} G}$
  ($=\sum_{i\in [G\backslash Q_0]} e_i*1_{{\mathbbm{k}} G_i}$) in the algebra
  $T_S(M)*G$.
\end{rem}

\begin{thm}
  \label{sec:defin-main-results-4}
  Let $Q_0$, $S$, $G$, and $M$ be as in
  Setting~\ref{sec:defin-notat-main-2}. Let $Q_G$ be as in
  Setting~\ref{sec:defin-notat-main-10}. Assume that
  $U= {\mathbbm{k}} G_i \cdot {\varepsilon}_U$ for all vertices $(i,U)$ of $Q_G$. Then,
  $\Xi$, as introduced in (\ref{eq:43}), is an isomorphism of
  ${\mathbbm{k}}$-algebras such that the following diagram is commutative,
  \begin{equation}
    \label{eq:45}
  \xymatrix@R=2ex@C=10ex{
    & \tilde e \cdot(T_S(M)*G)\cdot \tilde e \ar[dd]^\Xi \\
    {\mathbbm{k}} Q_G \ar[ru]^{~(\ref{eq:27})} \ar[rd]_{~(\ref{eq:32})} \\
    & \intw^{\mathrm{op}}\,.
  }
\end{equation}
  In particular, if (\ref{eq:27}) is used as an identification, then,
  for all $\theta\in \tilde e\cdot (T_S(M)*G)\cdot \tilde e$,
  \begin{equation}
    \label{eq:44}
    \theta = \sum_{\text{$\gamma$ path in $Q_G$}}
    (\Xi(\theta)|\varphi_\gamma)\cdot \gamma\,.
\end{equation}
\end{thm}

\begin{ex}
  \label{sec:defin-main-results-15}
  This is the continuation of
  Example~\ref{sec:defin-main-results-14}. Denote by $w$ the path
  $x_{\bullet,\overline{0}}x_{\overline{0},\overline{1}}x_{\overline{1},\bullet}$
  in $Q$, this is an oriented cycle of length three going clockwise in
  the picture of $Q$ in Example~\ref{sec:defin-main-results-8}. The
  present example describes the intertwiner $\Xi(\tilde e \cdot
  (w*\mathrm{Id}) \cdot \tilde e)$ as well as its preimage under the
  isomorphism ${\mathbbm{k}} Q_G \xrightarrow{(\ref{eq:32})}\intw^{{\mathrm{op}}}$.

  Because of the definition,
  \[
    \Xi(\tilde e \cdot (w*\mathrm{Id}) \cdot \tilde e) = \sum_{s,t\in
      \mathbb Z/2\mathbb Z} \Xi(\tilde e \cdot (w*\mathrm{Id}) \cdot
    \tilde e)_{s,t}\,,
  \]
  where, for all $s,t\in \mathbb Z/2\mathbb Z$, the piece of notation
  $\Xi(\tilde e \cdot (w*\mathrm{Id}) \cdot \tilde e)_{s,t}$ stands
  for the intertwiner
  $\rho_s \to M(\bullet,\overline{0},\overline{0},\bullet; \rho_t)$
  which maps ${\varepsilon}_s$ to
  $(e_\bullet\otimes {\varepsilon}_s)\cdot \tilde e \cdot (w\otimes \mathrm{Id})
  \cdot \tilde e \cdot (e_\bullet \otimes {\varepsilon}_t)$. For all
  $s,t\in \mathbb Z/2\mathbb Z$, the latter expression is equal to
  $(1_{{\mathbbm{k}} Q}\otimes {\varepsilon}_s)\cdot (w \otimes {\varepsilon}_t)$, which in turn is
  equal to
  $\frac{1}{2}(x_{\bullet,\overline{0}}x_{\overline{0},\overline{1}}x_{\overline{1},\bullet}
  \otimes {\varepsilon}_t +(-1)^{s+\overline{1}+t}
  x_{\bullet,\overline{1}}x_{\overline{1},\overline{0}}x_{\overline{0},\bullet}
  \otimes {\varepsilon}_t)$; therefore, using the notation of
  Example~\ref{sec:defin-main-results-11},
  \[
    \Xi(\tilde e \cdot (w*\mathrm{Id}) \cdot \tilde
    e)_{s,t}=F_{\overline{0}}\,.
  \]

  Now, let $s,t,a\in \mathbb Z/2\mathbb Z$; using the computation of
  the pairings of the shape $(F_a|\phi_b)$ made in
  Example~\ref{sec:defin-main-results-13} as well as the description
  of the intertwiners relative to $M^*$ associated to paths in $Q_G$
  made in Example~\ref{sec:defin-main-results-14}, and denoting by
  $\gamma$ the path
  $(\bullet,\rho_s) \xrightarrow{b_s} (\overline{0},{\mathbbm{k}}\cdot
  \mathrm{Id}) \xrightarrow{d_a} (\overline{0},{\mathbbm{k}}\cdot \mathrm{Id})
  \xrightarrow{c_t} (\bullet,\rho_t)$ of $Q_G$,
  \[
    (\Xi(\tilde e(w\otimes \mathrm{Id})\cdot \tilde e)_{s,t}
    |
    \varphi_\gamma) =
    \left\{
      \begin{array}{ll}
        \frac{1+(-1)^t}{2} & \text{if $a=\overline{0}$} \\
        0 & \text{if not;}
      \end{array}
      \right.
    \]
    Therefore
    $\Xi(\tilde e\cdot (w\otimes \mathrm{Id})\cdot \tilde e)_{s,t} =
    \frac{1+(-1)^t}{2}\cdot b_sd_{\overline{0}}c_t$
    (see~(\ref{eq:44})). As a whole,
    $\Xi(\tilde e\cdot (w\otimes \mathrm{Id})\cdot \tilde e) =
    b_{\overline{0}}d_{\overline{0}}c_{\overline{0}} +
    b_{\overline{1}}d_{\overline{0}}c_{\overline{0}}$.
\end{ex}
\section{Practical aspects}
\label{sec:examples}

Let $Q_0$, $S$, $G$ and $M$ be as in
Setting~\ref{sec:defin-notat-main-2}.

Applying Theorem~\ref{sec:defin-main-results-4} for computations
requires to compute $Q_G$ first. Actually, rather than computing
explicitly the intertwiners relative to $M$ which form the arrows of
$Q_G$, it is simpler to
\begin{enumerate}
\item compute a basis of $\Hom_{{\mathbbm{k}} G_j}(\tau,M^*(j,i;\rho))$ for all
  pairs $(i,\rho)$ and $(j,\tau)$ such that $i,j\in [G\backslash Q_0]$,
  $\rho\in \irr(G_i)$, and $\tau \in \irr(G_j)$.
\end{enumerate}
Because of the pairing $(-|-)$, this yields a quiver $Q_G$ such as in
Setting~\ref{sec:defin-notat-main-10} as well as the intertwiners
$\varphi_\alpha$ for all arrows $\alpha$ of $Q_G$. This avoids
computing explicitly the intertwiners corresponding to the arrows of
$Q_G$, which are not used in the formulas of
Theorem~\ref{sec:defin-main-results-4}.

Given an element $\theta$ of
$\tilde e \cdot (T_S(M)*G)\cdot \tilde e$, say homogeneous of degree
$n$ for the grading induced by the tensor powers of $M$, the
decomposition into a linear combination of paths in $Q_G$ of the
inverse image of $\theta$ under the isomorphism
${\mathbbm{k}} Q_G \xrightarrow{~\eqref{eq:27}} \tilde e \cdot (T_S(M)*G) \cdot
\tilde e$
may be computed as follows using
Theorem~\ref{sec:defin-main-results-4},
\begin{enumerate}
  \setcounter{enumi}{1}
\item compute $\Xi(\theta)$, which is a sum of intertwiners relative
  to $M$, see~\eqref{eq:43};
\item compute the intertwiners $\varphi_\gamma$ for all paths $\gamma$
  of length $n$ in $Q_G$, see Notation~\ref{sec:defin-main-results-1};
\item compute the pairing $(\Xi(\theta)|\varphi_\gamma)$ for all such
  paths $\gamma$, see Proposition~\ref{sec:defin-notat-main-13}.
\end{enumerate}
The desired linear combination of paths is hence
$\sum_{\gamma}(\Xi(\theta)|\varphi_\gamma) \cdot \gamma$ (see
\eqref{eq:44}). 

These computations may be performed by a computer that is able to
determine the irreducible representations of $\irr(G_i)$ for all
$i\in [G\backslash Q_0]$. This section illustrates these computations.

\subsection{How to pair intertwiners?}
\label{sec:how-pair-intertw}

In view of \eqref{eq:12}, the main difficulty in the computations
listed in the introduction of Section~\ref{sec:examples} lies in
computing pairings for $(-|-)$. This section explains how to compute
\eqref{eq:29} in combinatorial terms in the following specific
setting.

\begin{set}
  \label{sec:how-pair-intertw-1}
  Let $Q_0$, $S$, $G$ and $M$ be as in
  Setting~\ref{sec:defin-notat-main-2}. Let $Q$ be a quiver with
  vertex set being $Q_0$ and such that $M$ is equal to the vector
  space with basis elements being the arrows of $Q$. Let $Q_G$ be as
  in Setting~\ref{sec:defin-notat-main-10}.  Assume that
\begin{itemize}
\item the stabiliser $G_i$ is abelian for all $i\in Q_0$,
\item $U = {\mathbbm{k}} G_i \cdot {\varepsilon}_U$ for all vertices $(i,U)$ of $Q_G$, denote
  by $\chi_U$ the character $G_i\to {\mathbbm{k}}^\times$ of $U$,
\item and, for all $g\in G$ and all paths $\gamma$ in $Q$, there
  exists a scalar $\chi_{g,\gamma}\in {\mathbbm{k}}^\times$ and a path
  $g(\gamma)$ such that $^g\gamma = \chi_{g,\gamma}\,g(\gamma)$.
\end{itemize}
\end{set}

If $G$ is abelian, then it is possible to find a quiver $Q$ and to
choose irreducible representations $U$ of $G_i$, for all $i\in
[G\backslash Q_0]$, fitting in this setting.

\begin{notation}
  \label{sec:how-pair-intertw-3}
  The following notation relative to $M$ is useful.
  \begin{itemize}
  \item For all arrows $a\colon i\to j$ of $Q$, denote by $a^*$ the
    element of $M^*$ such that $a^*(a)=1$ and $a^*(b) = 0$ for all
    arrows $b$ of $Q$ distinct from $a$. 
  \item Denote by $Q^*$ the quiver with vertex set being $Q_0$ and whose
    arrows are the linear forms $a^*$ introduced just before. This
    quiver is isomorphic to the opposite quiver of $Q$.
  \item For all paths $\gamma \colon t_0 \xrightarrow{a_1} t_1 \to
    \cdots \to t_{n-1} \xrightarrow{a_n} t_n$ in $Q$, denote by
    $\gamma^*$ the path $t_n \xrightarrow{a_n^*} t_{n-1}\to
    \cdots \to t_1 \xrightarrow{a_1^*} t_0$ in $Q^*$.
    \end{itemize}
\end{notation}

Note that every path in $Q$ starting in some vertex $i_0$ lying in
$[G\backslash Q_0]$ is of the shape $i_0\to y_1\cdot i_1 \to \cdots
\to (y_1\cdots y_n)\cdot i_n$ for a unique sequence $i_1,\ldots,i_n$
in $[G\backslash Q_0]$ and a unique sequence $y_1,\ldots,y_n$ in $G$
such that $y_t\in [G/G_{i_t}]$ for all $t\in \{1,\ldots,n\}$.

Let $\underline i=i_0,\ldots,i_n$ be a sequence in
$[G\backslash Q_0]$, consider $U\in \irr(G_{i_0})$ and
$V\in \irr(G_{i_n})$, and let $f \colon U \to M(\underline i;V)$ and
$\varphi \colon V \to M^*(\underline i^o;U)$ be intertwiners. Since
the representation $V$ is one dimensional, then,
\begin{equation}
  \label{eq:49}
  f({\varepsilon}_u) = \sum_{\gamma}\alpha_\gamma \gamma \otimes y_1 y_2\cdots y_n{\varepsilon}_V\,,
\end{equation}
where $\gamma$ runs through all paths in $Q$ of the shape
\[
i_0 \to y_1 \cdot i_1 \to y_1y_2 \cdot i_2 \to \cdots \to (y_1\cdots
y_n) \cdot i_n
\]
for some (unique) sequence $\underline y = y_1,\ldots,y_n$ such that
$y_t\in [G/G_{i_t}]$ for all $t\in \{1,\ldots,n\}$, and
$\alpha_\gamma \in {\mathbbm{k}}$.  Note that the considered paths are precisely
the ones of the shape $i_0\to j_1\to \cdots \to j_n$ such that $j_t$
lies in the orbit of $i_t$ for all $t \in \{1,\ldots,n\}$. Similarly,
\begin{equation}
  \label{eq:50}
  \varphi({\varepsilon}_V) = \sum_{\gamma'} \beta_{\gamma'} \gamma' \otimes
  x_{n-1}\cdots x_0{\varepsilon}_U\,, 
\end{equation}
where $\gamma'$ runs through all paths in $Q$ of the shape
\[
i_n \to x_{n-1} \cdot i_{n-1} \to x_{n-1}x_{n-2} \cdot i_{n-2} \to
\cdots \to (x_{n-1} x_{n-2}\cdots x_0) \cdot i_0
\]
for some (unique) sequence $\underline x = x_{n-1},\ldots,x_0$
such that $x_t\in [G/G_{i_t}]$ for all $t\in \{0,\ldots,n-1\}$, and
$\beta_{\gamma'} \in {\mathbbm{k}}$.  

\begin{lem}
  \label{sec:how-pair-intertw-2}
  Using the decompositions \eqref{eq:49} and \eqref{eq:50}, then
  \begin{equation}
    \label{eq:51}
    (f | \varphi) = \sum_{\gamma}
    \alpha_\gamma \cdot \beta_{(y_1\cdots y_n)^{-1}(\gamma)^*} \cdot
    \chi_{(y_1\cdots y_n)^{-1},\gamma} \cdot \chi_U(h_0)\,, 
  \end{equation}
  where $\gamma$ runs through all paths in $Q$ of the shape
  $i_0 \to j_1 \to \cdots \to j_n$ such that $j_t$ lies in the orbit
  of $i_t$ for all $t\in \{1,\ldots,n\}$, where
  $\alpha_\gamma \in {\mathbbm{k}}$, and where, for all such paths $\gamma$,
    \begin{itemize}
    \item $\underline y = y_1,\ldots,y_n$ denotes the sequence in $G$
      such that $y_t\in [G/G_{i_t}]$, for all $t\in \{1,\ldots,n\}$
      and $\gamma$ is a path of the shape
      $i_0\to y_1\cdot i_1 \to \cdots \to y_1\cdots y_n \cdot i_n$,
    \item $h_0$ is the element of $G_{i_0}$ such as in (c) in
      Proposition~\ref{sec:defin-notat-main-13},
    \item $\chi_{(y_1\cdots y_n)^{-1},\gamma}$ is the element of
      ${\mathbbm{k}}^\times$ and $(y_1\cdots y_n)^{-1}(\gamma)$ is the path in
      $Q$ such that $^{(y_1\cdots y_n)^{-1}}\gamma = \chi_{(y_1\cdots
        y_n)^{-1},\gamma} \,(y_1\cdots y_n)^{-1}(\gamma)$,
    \item and $((y_1\cdots y_n)^{-1}(\gamma))^*$ is the reverse path
      in $Q^*$ associated to the path $(y_1\cdots y_n)^{-1}(\gamma)$
      in $Q$.
    \end{itemize}
\end{lem}
\begin{proof}
  For all sequences $\underline y=y_1,\ldots,y_n$ and
  $\underline x = x_{n-1},\ldots,x_0$ such that $x_t\in [G/G_{i_t}]$
  and $y_t\in [G/G_{i_t}]$ for all $t$, denote by ${\mathcal C}(\underline y)$
  and ${\mathcal C}'(\underline x)$ the sets of paths in $Q$ and in $Q^*$ of
  the shape
  $i_0\to y_1\cdot i_1\to \cdots \to y_1\cdots y_n \cdot i_n$ and
  $i_n\to x_{n-1} \cdot i_{n-1} \to \cdots \to x_{n-1}\cdots x_0\cdot
  i_0$, respectively.  For all $\underline y$,
  \[
  f_{\underline y}({\varepsilon}_U) = \sum_{\gamma \in {\mathcal C}(\underline y)}
  \alpha_\gamma \gamma \otimes y_1\cdots y_n {\varepsilon}_U\,;
  \]
  hence, the following term that serves in the definition of
  $(f|\varphi)$, see (\ref{eq:29}),
  \[
  ^{(y_1\cdots y_n)^{-1}}f_{\underline y}^{(1)}({\varepsilon}_U) \otimes \cdots
  \otimes \,^{(y_1\cdots y_n)^{-1}}f_{\underline y}^{(n)}({\varepsilon}_U)
  \otimes f_{\underline y}^{(0)}({\varepsilon}_U)\,,
  \]
  is equal to
  \[
  \sum_{\gamma \in {\mathcal C}(\underline y)} \alpha_\gamma \cdot \chi_{(y_1\cdots
    y_n)^{-1},\gamma} \cdot \underset{\in M^{\otimes_S
      n}}{\underbrace{(y_1\cdots y_n)^{-1}(\gamma)}} \otimes {\varepsilon}_U\,.
  \]
  Similarly, for all $\underline x$,
  \[
  \varphi^{\underline x}({\varepsilon}_V) = \sum_{\gamma'\in {\mathcal C}(\underline x)}
  \beta_{\gamma'} \gamma' \otimes x_{n-1}\cdots x_0{\varepsilon}_U\,.
  \]

  Therefore, given $\underline y$, the ``$\Pi_t$''-term in
  \eqref{eq:29} is equal to
  \begin{equation}
    \label{eq:52}
  \sum_{\gamma \in {\mathcal C}(\underline y),\,\gamma'\in {\mathcal C}(\underline x)}
  \alpha_{\gamma} \cdot \beta_{\gamma'}\, \chi_{(y_1\cdots
    y_n)^{-1},\gamma} \cdot
  b_n^*(c_n)\cdots b_2^*(c_2)\cdot b_1^*(c_1) \cdot\chi_U(h_0) \cdot{\varepsilon}_U\,,
\end{equation}
  where
  \begin{itemize}
  \item $\underline x = x_{n-1},\ldots,x_0$ and $h_0$ are the elements
    of $G$ determined by $\underline y$ and by (a), (b), and (c) in
    Proposition~\ref{sec:defin-notat-main-13},
  \item $b_1^*,\ldots,b_n^*$ are the arrows of $Q^*$ such that
    $\gamma'$ is
    $\cdot \xrightarrow{b_n^*}\cdot \xrightarrow{b_{n-1}^*}\cdots
    \xrightarrow{b_1^*}\cdot$,
  \item and $c_1,\ldots,c_n$ are the arrows of $Q$ such that
    $(y_1\cdots y_n)^{-1}(\gamma)$
    is $\cdot \xrightarrow{c_1}\cdot \xrightarrow{c_2}\cdots
    \xrightarrow{c_n}\cdot$.
  \end{itemize}
  By definition of the arrows $b_1^*,\ldots,b_n^*$, the summand of
  \eqref{eq:52} with index $\gamma$ is non-zero only if the paths
  $\cdot \xrightarrow{b_1}\cdot\xrightarrow{b_2} \cdots
  \xrightarrow{b_n}\cdot$
  and $(y_1\cdots y_n)^{-1}(\gamma)$ in $Q$ are equal, that is, only
  if $\gamma' = ((y_1\cdots y_n)^{-1}(\gamma))^*$.
  Summing~\eqref{eq:52} over all possible $\underline y$ yields that
  $(f |\varphi){\varepsilon}_U$ is equal to
  \[
  \sum_\gamma \alpha_\gamma \beta_{((y_1\cdots y_n)^{-1}(\gamma))^*}
  \chi_{(y_1\cdots y_n)^{-1},\gamma}\chi_U(h_0){\varepsilon}_U\,.
  \]
  This proves \eqref{eq:51}.
\end{proof}
\subsection{An example}
\label{sec:an-example}

In this example, $Q$ is the following quiver whose vertices are the
elements of $\mathbb Z/4\mathbb Z$
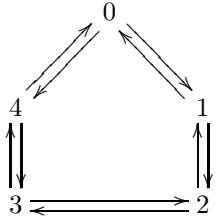
\begin{equation}
  \label{eq:54}
  \xymatrix{ & \overline{0} \ar@<2pt>[rd] \ar@<2pt>[ld] &
    \\
    \overline{4} \ar@<2pt>[ru] \ar@<2pt>[d] & & \overline{1}
    \ar@<2pt>[d] \ar@<2pt>[lu]
    \\
    \overline{3} \ar@<2pt>[u] \ar@<2pt>[rr] & & \overline{2}
    \ar@<2pt>[ll] \ar@<2pt>[u] }
\end{equation}
and $G$ is the dihedral group of order $10$
\[
G = \langle c,\,\tau\ |\ c^5\,, \tau^2 \,, \tau c\tau c\rangle\,.
\]
Denote by ${\varepsilon}$ the group homomorphism $G \to \{-1,1\}$ such that
${\varepsilon}(c)=1$ and ${\varepsilon}(\tau)=-1$. Denote the arrows of $Q$ by
$x_{i,j}\colon i \to j$. For convenience, denote the arrows of $Q^*$
by $x'_{i,j} \colon i\to j$, hence $(x_{j,i})^* = x'_{i,j}$. The
actions of $G$ on $Q_0$ and $M$ are assumed to be the ones such that,
for all $i\in Q_0$ and $g\in G$, and for all arrows $x_{i,j}$ of $Q$,
\[
c\cdot i = i+\overline{1},\ \tau \cdot i = -i,\ \,^gx_{i,j} =
{\varepsilon}(g) x_{g\cdot i,g\cdot j}\,.
\]
The purpose of this section is to make a full illustration of
Theorem~\ref{sec:defin-main-results-4}, that is, for all
$w\in T_S(M)*G$,
\begin{itemize}
\item to describe the intertwiner $\Xi(\theta)$, where $\theta$
  denotes $\tilde e \cdot w \cdot \tilde e$;
\item and to describe the preimage of $\theta$ under the isomorphism
  ${\mathbbm{k}} Q_G \xrightarrow{~\eqref{eq:27}} \tilde e \cdot (T_S(M)*G)\cdot
  \tilde e$.
\end{itemize}

Let $[G\backslash Q_0]$ be $\{\overline{0}\}$. Note that $G_{\overline{0}} =
\{\id,\tau\}$. Let $[G/G_{\overline{0}}]$ be $\{\id,c,c^2,c^3,c^4\}$. For all $s\in
\{0,1\}$, denote by ${\varepsilon}_s$ the following primitive idempotent of ${\mathbbm{k}}
G_{\overline{0}}$,
\[
{\varepsilon}_s = \frac{1}{2}(\id +(-1)^s\tau)
\]
and denote ${\varepsilon}_s\cdot {\mathbbm{k}} G_{\overline{0}}$ by $\rho_s$. Finally, let $\irr(G_{\overline{0}})$ be
equal to $\{\rho_0,\rho_1\}$.  In particular,
\[
\tilde e = e_{\overline{0}}*{\varepsilon}_0 + e_{\overline{0}}*{\varepsilon}_1\,.
\]

The setting presented below is assumed until the end of the
subsection. In view of the purpose stated previously, this is not a
loss of generality in making this assumption because
\begin{itemize}
\item the elements of the shape $p*g$, where $p$ is a path in $Q$ and
  $g\in G$, form a basis of $T_S(M)*G$ as a vector space;
\item every $g\in G$ is equal to $c^\ell\tau^r$ for a unique couple
  $(\ell,r) \in \{0,1,2,3,4\}\times \{0,1\}$;
\item for all paths $p$ of $Q$ and for all $(\ell,r) \in
  \{0,1,2,3,4\}\times \{0,1\}$, the product $\tilde e \cdot (p * g)
  \cdot \tilde e$ vanishes as soon as $p$ is not a path from $\overline
  0$ to $-\overline{\ell}$.
\end{itemize}
\begin{set}
  \label{sec:an-example-1}
  Assume that $w$ is equal to $p*c^\ell\tau^r$, where $(\ell,r) \in
  \{0,1,2,3,4\} \times \{0,1\}$ and $p$ is a path in $Q$ from
  $\overline 0$ to $-\overline{\ell}$. Denote the length of $p$ by
  $n$. Denote by $i_0,\ldots,i_n$ the sequence of vertices of $Q$ such
  that $p$ has the shape $i_0\to i_1\to \cdots \to i_n$. Thus, $p =
  x_{i_0,i_1}x_{i_1,i_2}\cdots x_{i_{n-1},i_n}$. In particular,
  $i_0=\overline{0}$,  $i_n = -\overline{\ell}$ and $i_k-i_{k-1}\in
  \{-\overline{1},\overline{1}\}$ for all $k\in \{1,\ldots,n\}$.
\end{set}

\subsubsection{Computation of $Q_G$}
\label{sec:computation-q_g}

By definition, for all $t\in \{0,1\}$,
\[
M^*(\overline{0},\overline{0};\rho_t) = \mathrm{Span}(x'_{\overline{0},\overline{1}}\otimes
c{\varepsilon}_t,\,x'_{\overline{0},\overline{4}}\otimes c^4{\varepsilon}_t)\,,
\]
where the action of $\tau$ is given as follows
\[
\left\{
  \begin{array}{rcl}
    x'_{\overline{0},\overline{1}}\otimes c{\varepsilon}_t & \longmapsto & (-1)^{1+t}x'_{\overline{0},\overline{4}}\otimes c^4{\varepsilon}_t \\
    x'_{\overline{0},\overline{4}}\otimes c^4{\varepsilon}_t & \longmapsto & (-1)^{1+t}x'_{\overline{0},\overline{1}}\otimes c{\varepsilon}_t\,.
  \end{array}
\right.
\]
Hence, for all $s,t \in \{0,1\}$,
\[
\Hom_{{\mathbbm{k}} G_{\overline{0}}}(\rho_s,M^*(\overline{0},\overline{0};\rho_t)) = \mathrm{Span}(\varphi_{s,t})\,,
\]
where $\varphi_{s,t}$ is the intertwiner such that
$\varphi_{s,t}({\varepsilon}_s) = {\varepsilon}_s \cdot (x'_{\overline{0},\overline{1}} \otimes c{\varepsilon}_t)$
($= {\varepsilon}_s \cdot (x_{\overline{0},\overline{1}}' \otimes c) \cdot {\varepsilon}_t$), that is,
\[
\varphi_{s,t}({\varepsilon}_s) = x'_{\overline{0},\overline{1}} \otimes c {\varepsilon}_t +
(-1)^{s+t+1}x'_{\overline{0},\overline{4}}\otimes c^4 {\varepsilon}_t\,.
\]

Denote by $f_{s,t}$ the intertwiner $\rho_s \to M(\overline{0},\overline{0};\rho_t)$ such
that $(f_{s,t}|\varphi_{t,s})=1$, for all $s,t\in \{0,1\}$. Then $Q_G$
may be taken equal to the following quiver
\[
\xymatrix{
  (1,\rho_0) \ar@(ul,ur)^{f_{0,0}} \ar@<2pt>[rr]^{f_{0,1}} &&
  (1,\rho_1) \ar@(ul,ur)^{f_{1,1}} \ar@<2pt>[ll]^{f_{1,0}}\,.
}
\]

\subsubsection{The intertwiner $\Xi(\theta)$}
\label{sec:intertw-assoc-theta}

By definition, $\Xi(\theta)$ decomposes as
\begin{equation}
  \label{eq:48}
\Xi(\theta) = \sum_{s,t\in \{0,1\}} \Xi(\theta)_{s,t}\,,
\end{equation}
where, for all $s,t\in \{0,1\}$, the intertwiner
$\Xi(\theta)_{s,t} \colon \rho_s \to
M(\underset{n+1}{\underbrace{0,\ldots,0}};\rho_t)$ is such that
$ \Xi(\theta)_{s,t}({\varepsilon}_s) = {\varepsilon}_s \cdot \tilde e \cdot w \cdot \tilde
e\cdot {\varepsilon}_t$, see (\ref{eq:43}). Note that
\begin{equation}
  \label{eq:53}
  \begin{array}{rcl}
    \Xi(\theta)_{s,t}({\varepsilon}_s)
    & = &
          {\varepsilon}_s \cdot \tilde e \cdot w \cdot \tilde
          e\cdot {\varepsilon}_t \\
  & = &
        \frac{1}{2}(\mathrm{Id}+(-1)^s\tau) \cdot (p \otimes c^\ell \tau^r)\cdot {\varepsilon}_t \\
  & = &
        \frac{(-1)^{rt}}{2}
        (\mathrm{Id}+(-1)^s\tau)
        \cdot (p \otimes c^\ell {\varepsilon}_t)
        \\
  & = &
        \frac{(-1)^{rt}}{2} \left(
        p \otimes c^\ell {\varepsilon}_t
        + (-1)^{s+t} (\,^{\tau}p)\otimes c^{-\ell}{\varepsilon}_t
        \right)\,,
\end{array}
\end{equation}
where
$^\tau p = (-1)^nx_{-i_0,-i_1}x_{-i_1,-i_2}\cdots x_{-i_{n-1},-i_n}$.

\subsubsection{Intertwiners relative to $M^*$ associated to paths of
  $Q_G$}
\label{sec:intertw-relat-m}

Consider a path of length denoted by $m$ in $Q_G$, say
$q \colon (0,\rho_{s_0}) \to (0,\rho_{s_1}) \to \cdots \to
(0,\rho_{s_m})$, where $s_0,\ldots,s_m\in \{-1,1\}$. Then
$\varphi_q$ is the following intertwiner
$\rho_{s_m}\to M^*(\underset{n+1}{\underbrace{0,\ldots,0}};\rho_{s_0})$,
\[
  \varphi_q = \varphi_{s_1,s_0} \circledast \varphi_{s_2,s_1}
  \circledast \cdots \circledast \varphi_{s_m,s_{m-1}}\,.
\]
Consider $\varphi_q({\varepsilon}_{s_m})$ as an element of ${\mathbbm{k}} Q^**G$;
following (\ref{eq:12}) it equals the product
$(x_{\overline{0},\overline{1}}'\otimes
(-1)^{s_m+s_{m-1}+1}x_{\overline{0},\overline{4}} \otimes c^4)\cdots
(x_{\overline{0},\overline{1}}'\otimes
(-1)^{s_1+s_0+1}x_{\overline{0},\overline{4}} \otimes c^4)$.  Hence
\begin{equation}
  \label{eq:46}
  \begin{array}{l}
    \varphi_q({\varepsilon}_{s_m}) =
    \sum
    (-1)^{\sum s_k+s_{k-1}+1}
    x_{\overline{0},j_{m-1}}' x_{j_{m-1}j_{m-2}}' \cdots x_{j_1,j_0}'
    \otimes c^{j_0}\,,
  \end{array}
\end{equation}
where 
\begin{itemize}
\item the first sum runs over all sequences
  $j_{m-1},\ldots,j_0\in \mathbb Z/4\mathbb Z$ such that the
  expression $x_{\overline{0},j_{m-1}}' \cdots x_{j_1,j_0}'$ is a path
  in $Q'$, which means that
  $j_t-j_{t-1}\in \{-\overline{1},\overline{1}\}$ for all
  $t\in \{1,\ldots,m\}$ (taking $j_m=\overline{0}$),
\item and the second sum (in the exponent) runs over all indices
  $k\in \{1,\ldots,m\}$ such that $x_{j_k,j_{k-1}}'$ of $Q'$ goes
  anticlockwise when $Q'$ is drawn like $Q$ in (\ref{eq:54}), which
  means that $j_k-j_{k-1}=\overline{1}$.
\end{itemize}
      
\subsubsection{Computation of the pairing $(\Xi(\theta)|\varphi_q)$}
\label{sec:pair-xith}

Let
$q \colon (0,\rho_{s_0}) \to (0,\rho_{s_1}) \to \cdots \to
(0,\rho_{s_m})$ be a path in $Q_G$ like in
\ref{sec:intertw-relat-m}. Because of the definition of the action of
$G$ on ${\mathbbm{k}} Q$, the pairing $(\Xi(\theta)|\varphi_q)$
($=(\Xi(\theta)_{s_0,s_m}|\varphi_q)$) is given by (\ref{eq:51}). In
particular, because of (\ref{eq:53}) and (\ref{eq:46}),
\begin{itemize}
\item this pairing vanishes if $m\neq n$, that is, $p$ and $q$ have
  different lengths,
\item if $m=n$ then, for all paths $\gamma$ distinct from both
  $x_{i_0,i_1} x_{i_1,i_2} \cdots x_{i_{n-1},i_n}$ and
  $x_{i_0,-i_1} x_{-i_1,-i_2}\cdots x_{-i_{n-1},-i_n}$, the
  term of index $\gamma$ in (\ref{eq:53}) vanishes.
\end{itemize}
Using (\ref{eq:53}) and (\ref{eq:46}), the scalars appearing in
(\ref{eq:51}) may be computed directly for these two possibilities for
$\gamma$, assuming that $m=n$. Table~\ref{tab:1} describes these
scalars, they appear in bold face. Recall that $i_0=\overline{0}$.

\begin{table}[!ht]
  \tiny
\begin{tabularx}{\textwidth}{>{\hsize=.17\hsize}X>{\hsize=.39\hsize\centering}X>{\hsize=.37\hsize\centering\arraybackslash}X}
  \hline $\gamma$ &
  $x_{i_0,i_1} x_{i_1,i_2} \cdots x_{i_{n-1},i_n}$ &
  $x_{i_0,-i_1}
  x_{-i_1,-i_2}\cdots x_{-i_{n-1},-i_n}$ \\
  \hline
  $\mathbf{\alpha_\gamma}$ & $\mathbf{\frac{(-1)^{rs_n}}{2}}$ &
  $\mathbf{\frac{(-1)^{rs_n+s_0+s_n+n}}{2}}$ \\
  $y_1,\ldots,y_n$ &
  $c^{i_1-i_0},c^{i_2-i_1},\ldots,c^{i_n-i_{n-1}}$
  &
  $c^{-i_1+i_0},c^{-i_2+i_1},\ldots,c^{-i_n+i_{n-1}}$
  \\
  $y_1\cdots y_n$ & $c^{i_n}$ & $c^{-i_n}$ \\
  $^{(y_1\cdots y_n)^{-1}}\gamma$ &
  $x_{-i_n,i_1-i_n}x_{i_1-i_n,i_2-i_n}\cdots x_{i_{n-1}-i_n,\overline{0}}$
  &
  $x_{i_n,i_n-i_1}x_{i_n-i_1,i_n-i_2}\cdots x_{i_n-i_{n-1},\overline{0}}$
  \\
  $\mathbf{\chi_{(y_1\cdots y_n)^{-1},\gamma}}$ & $\mathbf{1}$ & $\mathbf{1}$ \\
  $(y_1\cdots y_n)^{-1}(\gamma)^*$ &
  $x_{\overline{0},i_{n-1}-i_n}'\cdots x_{i_2-i_n,i_1-i_n}'x_{i_1-i_n,-i_n}'$
  &
  $x_{\overline{0},i_n-i_{n-1}}'\cdots x_{i_n-i_2,i_n-i_1}'x_{i_n-i_1,i_n}'$
  \\ & & \\
  $\mathbf{\beta_{(y_1\cdots y_n)^{-1}(\gamma)^*}}$ & $\mathbf{(-1)^x}$ where $x$ is & $\mathbf{(-1)^x}$,
  where $x$ is \\
  &
  $\sum\limits_{\begin{array}{c}
                    1\leqslant k\leqslant n\\
                    \mathrm{s.t.}\ i_k-i_{k-1}=\overline{1}
                  \end{array}}s_k+s_{k-1}+1$
                & 
  $\sum\limits_{\begin{array}{c}
                    1\leqslant k\leqslant n\\
                    \mathrm{s.t.}\ i_k-i_{k-1}=-\overline{1}
                  \end{array}}s_k+s_{k-1}+1$
                \\ & & \\
                $h_0$ & $\mathrm{Id}$ & $\mathrm{Id}$ \\
                $\mathbf{\chi_{\rho_{s_0}}(h_0)}$ & $\mathbf{1}$ & $\mathbf{1}$ \\
  \hline
\end{tabularx}
\caption{\tiny The coefficients in the expression (\ref{eq:51}) of
  $(\Xi(\theta)_{s_0,s_m}|\varphi_q)$ according to $\gamma$}
\label{tab:1}
\end{table}
Thus
\[
  \begin{array}{rcl}
    (\Xi(\theta)|\varphi_q)
    & = &
          \frac{(-1)^{rs_0}}{2}(-1)^{
          \sum_{i_k-i_{k-1}=\overline{1}}
          (s_k+s_{k-1}+1)
          }
      \\
    & & + \\
    & &
        \frac{(-1)^{rs_0+s_0+s_n+n}}{2}(-1)^{
        \sum_{i_k-i_{k-1}=-\overline{1}}
          (s_k+s_{k-1}+1)}\,.
  \end{array}
\]

As whole (after factoring and simplifying),
\begin{equation}
  \label{eq:55}
  (\Xi(\theta)|\varphi_q) =
  \left\{
    \begin{array}{ll}
      (-1)^{rs_0}(-1)^{
    \sum_{i_k-i_{k-1}=\overline{1}}
    (s_k+s_{k-1}+1)} & \text{if $m=n$} \\
      0 & \text{otherwise.}
    \end{array}
    \right.
\end{equation}

\subsubsection{Decomposition of $\theta$ into paths in $Q_G$}
\label{sec:decomposition-theta}

Because of Theorem~\ref{sec:defin-main-results-4}, the preimage of
$\theta$ under the isomorphism
${\mathbbm{k}} Q_G \xrightarrow{~\eqref{eq:27}} \tilde e \cdot (T_S(M)*G)\cdot
\tilde e$ is equal to
\[
  \sum_{\text{$q$ path in $Q_G$}}
  (\Xi(\theta)|\varphi_q)\cdot q\,.
\]
In view of (\ref{eq:55}), this preimage is hence equal to
\[
  \sum_{s_0,\ldots,s_n\in \{0,1\}} (-1)^{rs_0}(-1)^{
    \sum_{i_k-i_{k-1}=\overline{1}} s_k+s_{k-1}+1 }\cdot
  ((0,\rho_{s_0})\to (0,\rho_{s_1})\to \cdots \to (0,\rho_{s_n}))\,.
\]

\subsubsection{Illustration}
\label{sec:illustration}

Assume that
$w=x_{\overline{0},\overline{1}}x_{\overline{1},\overline{2}}x_{\overline{2},\overline{3}}x_{\overline{3},\overline{4}}x_{\overline{4},\overline{0}}
* \mathrm{Id}
- x_{\overline{0},\overline{4}} x_{\overline{4},\overline{3}}
x_{\overline{3},\overline{2}} x_{\overline{2},\overline{1}}
x_{\overline{1},\overline{0}}*\mathrm{Id}$. Applying \ref{sec:decomposition-theta}
yields that the preimage of $\tilde e\cdot w\cdot \tilde e$
under the isomorphism
${\mathbbm{k}} Q_G \xrightarrow{~\eqref{eq:27}} \tilde e \cdot (T_S(M)*G) \cdot
\tilde e$ is equal to $-2\sum_{\gamma}\gamma$, where $\gamma$ runs
through all oriented cycles of length $5$ in $Q_G$.

\section{Details on the monoidal category of bimodules over
  groups}
\label{sec:pair-betw-morph}

This section details some isomorphisms and properties in the monoidal
category $(\mathrm{mod}(A^e),\otimes_A)$ and related to the operation, say
$\times$, and the pairing $\langle-|-\rangle$ considered in
Remark~\ref{sec:defin-main-results-2}.
\begin{itemize}
\item Section~\ref{sec:tensor-product-duals} describes the left dual
  of a tensor product of objects in $\mathrm{mod}(A^e)$.
\item Section~\ref{sec:adjunctions} details the the adjunction
  $(X\otimes_A-)\vdash (X^*\otimes_A-)$ for a given $X\in \mathrm{mod}(A^e)$.
\item Section~\ref{sec:pairing} details the compatibility of
  $\langle-|-\rangle$. 
\end{itemize}
The existence of the isomorphisms and adjunctions discussed in these
sections may be part of the folklore (see
\cite[Exercise 2.10.16]{MR3242743}). However, the proof of
Theorem~\ref{sec:defin-main-results} is based on their description,
which is the reason for detailing these here. For the sake of
simplicity, the presentation is made independently of the monoidal
category $(\mathrm{mod}(A^e),\otimes_A)$ and of the framework of
Theorem~\ref{sec:defin-main-results}.

Here are some conventions. In any group, $e$ denotes the neutral
element; and $\Hom_{{\mathbbm{k}} G}$ and $\otimes_{{\mathbbm{k}} G}$ are denoted by
$\Hom_G$ and $\otimes_G$, respectively, for all finite groups $G$.

  Given finite groups $G$ and $H$, by ``\emph{a module $_GX_H$}'' is
meant a ${\mathbbm{k}} G-{\mathbbm{k}} H$-bimodule, by ``\emph{a module $_GX$}'' is meant a
left ${\mathbbm{k}} G$-module, and by ``\emph{by a module $X_H$} is meant a right
${\mathbbm{k}} H$-module.

  Given finite groups $G$ and $H$ and given a ${\mathbbm{k}} G-{\mathbbm{k}} H$-bimodule
$X$, the dual vector space $X^*$ is a ${\mathbbm{k}} H-{\mathbbm{k}} G$-bimodule in a
natural way ($h\cdot \phi \cdot g = \phi(g\cdot \bullet \cdot h)$ for
all $g\in G$, $h\in H$ and $\phi \in X^*$). In particular, taking $H$
to be the trivial group, this defines a functor $X\mapsto X^*$ from
left ${\mathbbm{k}} G$-modules to right ${\mathbbm{k}} G$-modules.

\subsection{Tensor product of duals and dual of a tensor product}
\label{sec:tensor-product-duals}

\begin{lem}
  \label{sec:prop-psi-langle}
  Let $H$ be a finite group with order not divisible by $\mathrm{char}({\mathbbm{k}})$. For
  all finite dimensional ${\mathbbm{k}} H$-modules $X_H$ and $_HY$, there is a
  functorial bijective mapping
  \begin{equation}
    \label{eq:9}
    \begin{array}{rcl}
      Y^* \otimes_H X^* & \to & (X \otimes_H Y)^* \\
      \phi_2 \otimes \phi_1 & \mapsto & (x \otimes y \mapsto
                                        \sum_{h\in H} \phi_1(x\cdot h)
                                        \phi_2(h^{-1}\cdot y))\,.
    \end{array}
  \end{equation}
\end{lem}
\begin{proof}
  The mapping is well defined. When $X={\mathbbm{k}} H_H$ and $Y=\,_H{\mathbbm{k}} H$, it
  identifies with the following one, where $\{\delta_h\}_{h \in H}$
  denotes the canonical basis of $({\mathbbm{k}} H)^*$,
  \[
  \begin{array}{ccc}
    ({\mathbbm{k}} H)^* \otimes_H ({\mathbbm{k}} H)^* & \to & ({\mathbbm{k}} H)^* \\
    \delta_{h_2} \otimes \delta_{h_1} & \mapsto & \delta_{h_1h_2}\,;
  \end{array}
  \]
  and the latter is bijective. Since $X_H$ and $_HY$ are finite dimensional
  and projective, then (\ref{eq:9}) is bijective.
\end{proof}
The previous lemma generalises as follows. Given an integer
$n\geqslant 2$, a sequence of groups $G_0,\ldots,G_n$ and a finite
dimensional ${\mathbbm{k}} G_{i-1}-{\mathbbm{k}} G_i$-bimodule $U_i$ for every
$i\in \{1,\ldots,n\}$, there is a functorial bijective mapping
\begin{equation}
  \label{eq:15}
  X_n^* \otimes_{G_{n-1}} \cdots \otimes_{G_1} X_1^*
  \to
  (X_1 \otimes_{G_1}\cdots \otimes_{G_{n-1}} X_n)^*\,,
\end{equation}
which maps any  $\phi_n \otimes \cdots \otimes \phi_1$ to the
following linear form on $X_1 \otimes_{G_1}\cdots \otimes_{G_{n-1}} X_n$,
\[
x_1\otimes \cdots \otimes x_n \mapsto \sum_{g_1,\ldots,g_{n-1}} \prod_{t=1}^n
\phi_t(g_{t-1}^{-1}\cdot x_t \cdot g_t)\,,
\]
where $g_t$ runs through $G_t$  ($1\leqslant t \leqslant n-1$) and
$g_0=g_n=e$. 

\subsection{Adjunctions}
\label{sec:adjunctions}

\begin{lem}
  \label{sec:pair-betw-morph-1}
  Let $G$ and $H$ be finite groups with orders not divisible by
  $\mathrm{char}({\mathbbm{k}})$. Let $_GX_H$, $_HV$ and $_GU$ be (bi)modules over ${\mathbbm{k}} G$ and
  ${\mathbbm{k}} H$ such that $X$ and $V$ are finite dimensional. Then, there
  exists a functorial isomorphism
  \[
  \Psi^{V,U}_X \colon \Hom_H(V,X^* \otimes_G U) \to \Hom_G(X \otimes_H
  V,U)\,,
  \]
  such that, for all $\phi \in \Hom_H(V,X^* \otimes_G U)$, $x \in X$
  and $v\in V$,
  \begin{equation}
    \label{eq:16}
    (\Psi^{V,U}_X(\phi))(x \otimes v) = \sum_{g\in G}
    (\phi^{(2)}(v)) (g^{-1}\cdot x) g \cdot \phi^{(1)}(v)\,,
  \end{equation}
  where the following notation is used with sum sign omitted, for all
  $v\in V$,
  \[
  \phi(v) = \phi^{(2)}(v) \otimes \phi^{(1)}(v) \in X^*
  \otimes_G U\,.
  \]
  In particular, the pair of functors $(X\otimes_H-,X^*\otimes_G -)$
  between finite dimensional left ${\mathbbm{k}} H$-modules and finite
  dimensional left ${\mathbbm{k}} G$-modules is adjoint.
\end{lem}
\begin{proof}
  Using the definition of the structure of $H-G$-bimodule of $X^*$,
  simple changes of variable in the sum ``$\sum_{g\in G}$" show that
  $\Psi_X^{V,U}$ is indeed well defined.
  
  Since $_HV$ is finitely generated and projective, it suffices to
  prove the lemma assuming that $V={\mathbbm{k}} H$. In this case, $\Psi^{V,U}_X$
  identifies with the mapping $\lambda_{X,U}$ defined (for all left
  ${\mathbbm{k}} G$-modules $X$ and $U$) by
  \begin{equation}
    \label{eq:31}
    \begin{array}{rcl}
      X^* \otimes_G U & \to & \Hom_G(X,U) \\
      \phi \otimes u & \mapsto & (x \mapsto \sum_{g\in
                                 G}\phi(g^{-1}\cdot x)g\cdot u)\,.
    \end{array}
\end{equation}
Now, $\lambda_{{\mathbbm{k}} G,{\mathbbm{k}} G}$ identifies with the following bijective
mapping (recall that $\vert G\vert \in {\mathbbm{k}}^\times$).
  \[
  \begin{array}{ccc}
    {\mathbbm{k}} G^* \otimes_{G} {\mathbbm{k}} G & \to & {\mathbbm{k}} G \\
    \phi \otimes e & \mapsto & \sum_{g\in G}\phi(g^{-1})g\,.
  \end{array}
  \]
  Since $_GX$ is finite dimensional and projective,  and since $_GU$ is
  projective, it follows that $\lambda_{X,U}$ is bijective. Thus,
  $\Psi^{V,U}_X$ is bijective.
\end{proof}

\begin{lem}
  \label{sec:pair-betw-morph-3}
  Let $G,H,K$ be finite groups with orders not divisible by
  $\mathrm{char}({\mathbbm{k}})$. Let $_GX_H$, $_HY_K$, $_KW$ and $_GU$ be finite dimensional
  modules. Then, the following diagram commutes
  \begin{equation}
    \label{eq:33}
  \xymatrix{ \Hom_K(W,Y^* \otimes_H X^* \otimes_G U) \ar[rr]
    \ar[d]_{\Psi_Y^{W,X^* \otimes_G U}} && \Hom_K(W, (X \otimes_H Y)^* \otimes_G U)
    \ar[d]^{\Psi_{X\otimes_H Y}^{W,U}} \\
    \Hom_H(Y \otimes_K W,X^* \otimes_G U) \ar[rr]_{\Psi_X^{Y
        \otimes_K W,U}} && \Hom_G(X
    \otimes_H Y \otimes_K W,U)\,, }    
  \end{equation}
  where the top horizontal arrow is obtained upon applying
  $\Hom_K(W,-\otimes_G U)$ to (\ref{eq:9}).
\end{lem}
\begin{proof}
  In view of the compatibility of the involved mappings with direct
  sum decompositions of the finite dimensional projective modules
  $_KW$ and $_GU$, it is sufficient to prove the lemma assuming that
  $W={\mathbbm{k}} K$ and $U={\mathbbm{k}} G$.  In this case, the given diagram identifies
  with the following one
  \begin{equation}
    \label{eq:1}
    \xymatrix{
      Y^* \otimes_H X^*
      \ar[rr]^{~(\ref{eq:9})} \ar[d]_{\beta} &&
      (X \otimes_H Y)^*\ar[d]^{\alpha} \\
      \Hom_H(Y,X^*)
      \ar[rr]_{\gamma} &&
      \Hom_G(X \otimes_H Y ,{\mathbbm{k}} G)\,,
    }
  \end{equation}
  where, after using~(\ref{eq:31}),
  \begin{itemize}
  \item $\alpha$ is given by $\phi \mapsto ( x \otimes y \mapsto
    \sum_{g\in G} \phi(g^{-1}\cdot x \otimes y)g)$,
  \item $\beta$ is given by $\phi_2 \otimes \phi_1 \mapsto (y \mapsto
    \sum_{h\in H} \phi_2(h^{-1}\cdot y)\underset{\phi_1(\bullet
      h)}{\underbrace{h\cdot \phi_1}})$ and
  \item $\gamma$ is given by $f \mapsto (x \otimes y\mapsto \sum_{g\in
      G}(f(y))(g^{-1}\cdot x)g)$.
  \end{itemize}
  It is direct to check that (\ref{eq:1}) commutes, which proves the lemma.
\end{proof}

\subsection{Pairing of morphisms}
\label{sec:pairing}

\begin{defn}
  \label{sec:pairing-morphisms}
  Let $G$ and $H$ be finite groups whose orders are not divisible by
  $\mathrm{char}({\mathbbm{k}})$. Let $_GX_H$, $_HV$, and $_GU$ be (bi)modules such that $X$
  and $V$ are finite dimensional.  Assume that $U$ is simple. Define a
  pairing
\begin{equation}
  \label{eq:2}
  \begin{array}{crcl}
    \langle - | - \rangle \colon 
    & \Hom_G(U, X \otimes_H V) \otimes_{\mathbbm{k}} \Hom_H(V,X^*\otimes_G
      U) & \to & {\mathbbm{k}} \\
  \end{array}
\end{equation}
as follows. For all $f\in \Hom_G(U, X \otimes_H V)$ and
$\phi \in \Hom_H(V,X^*\otimes_G U)$,
\begin{equation}
  \label{eq:34}
\langle f | \phi \rangle \cdot \id_U = \Psi_X^{V,U}(\phi) \circ f
\end{equation}
(following Lemma~\ref{sec:adjunctions}, the pairing
$\langle-|-\rangle$ is non-degenerate).
\end{defn}

This pairing is compatible with the following operation on morphisms.
\begin{defn}
  \label{sec:pairing-morphisms-1}
  Let $G$, $H$ and $K$ be finite groups whose orders are not
  divisible by $\mathrm{char}({\mathbbm{k}})$. Let $_GX_H$, $_HY_K$, $_GU$, $_HV$ and $_KW$
  be (bi)modules. For all $f_1 \in \Hom_G(U, X \otimes_H V)$ and
  $f_2 \in \Hom_H(V, Y \otimes_K W)$, define
  $f_2 \times f_1 \in \Hom_G(U, X \otimes_H Y \otimes_K W)$ as the
  following composite morphism
\begin{equation}
  \label{eq:3}
  f_2 \times f_1  \colon U \xrightarrow{f_1} X \otimes_H V \xrightarrow{(X \otimes_H f_2)}X \otimes_H Y \otimes_K W\,.
\end{equation}
\end{defn}

Here is the above-mentioned compatibility.

\begin{lem}
  \label{sec:pair-betw-morph-4}
  Let
  \begin{itemize}
  \item $G$, $H$, and $K$ be finite groups with orders not divisible
    by $\mathrm{char}({\mathbbm{k}})$,
  \item $_GX_H$, $_HY_K$, $_GU$, $_HV$ and $_KW$ be (bi)modules such
    that $X$, $Y$, and $W$ are finite dimensional and $U$ and $V$ are
    simple,
  \item $f_1 \in \Hom_G(U, X \otimes_H V)$ and
    $f_2 \in \Hom_H(V, Y \otimes_K W)$ and
  \item $\phi_1 \in \Hom_H(V, X^*\otimes_G U)$ and
    $\phi_2 \in \Hom_K(W, Y^* \otimes_H V)$.
  \end{itemize}
  Denote by
  $\lambda$ the following isomorphism obtained upon applying
  $-\otimes_GU$ to (\ref{eq:9}),
  \[
  Y^* \otimes_H X^* \otimes_G U\xrightarrow{\sim} (X\otimes_HY)^*
  \otimes_G U\,.
  \]
  Then,
  $\langle f_2\times f_1 | \lambda \circ (\phi_1 \times \phi_2)\rangle =
  \langle f_2 | \phi_2 \rangle \cdot \langle f_1 |
  \phi_1\rangle$.
\end{lem}
\begin{proof}
  Using (\ref{eq:33}) yields that
  \[
  \Psi_{X \otimes_H Y}^{W,U}(\lambda \circ (\phi_1\times \phi_2)) =
  \Psi_X^{Y \otimes_K W,U}\circ \Psi_Y^{W,X^*\otimes_G U}(\phi_1\times
  \phi_2)\,.
  \]
  Recall that $\phi_1\times
  \phi_2$ is the following composite morphism
  \[
  W \xrightarrow{\phi_2} Y^* \otimes_H V
  \xrightarrow{Y^*\otimes_H\phi_1} Y^* \otimes_H X^* \otimes_G U\,.
  \]
  Since $\Psi_X^{\bullet,\bullet}$ and $\Psi_Y^{\bullet,\bullet}$ are
  adjunction isomorphisms, then
  \[
  \begin{array}{rcl}
    \Psi_X^{Y \otimes_K W,U}\circ \Psi_Y^{W,X^* \otimes_G U}(\phi_1\times \phi_2)
    & = &
          \Psi_X^{Y \otimes_K W,U}(\phi_1 \circ \Psi_Y^{W,V}(\phi_2))
    \\
    & = &
          \Psi_X^{V,U}(\phi_1) \circ (X \otimes_H
          \Psi_Y^{W,V}(\phi_2))\,.
  \end{array}
  \]
  Therefore,
  \[
  \begin{array}{rcl}
    \Psi_{X\otimes_H Y}^{W,U}(\lambda \circ (\phi_1 \times \phi_2)) \circ
    (f_2 \times f_1)
    &  = &
           \Psi_X^{V,U}(\phi_1) \circ
           (X \otimes_H \Psi_Y^{W,V}(\phi_2)) \circ \\
    &    &
           (X \otimes_H f_2) \circ f_1 \\
    & = &
          \Psi_X^{V,U}(\phi_1) \circ
          ( X\otimes_H \langle f_2 |\phi_2\rangle \cdot \id_V) \circ
          f_1 \\
    & = &
          \langle f_1|\phi_1\rangle \cdot \langle
          f_2|\phi_2\rangle\cdot \id_U\,.
  \end{array}
  \]
  The conclusion of the lemma then follows from the definition of
  $\langle -|-\rangle$.
\end{proof}
The previous lemma has the following generalisation. Consider
\begin{itemize}
\item an integer $n\geqslant 2$,
\item finite groups $G_0,\ldots,G_n$ with orders not divisible by
  $\mathrm{char}({\mathbbm{k}})$,
\item a finite dimensional left ${\mathbbm{k}} G_i$-module $U_i$ for every
  $i\in \{0,\ldots,n\}$, such that $U_0,\ldots,U_{n-1}$ are simple,
\item a finite dimensional ${\mathbbm{k}} G_{i-1}-{\mathbbm{k}} G_i$-bimodule $X_i$
  for every $i\in \{1,\ldots,n\}$,
\item $f_i \in \Hom_{G_{i-1}}(U_{i-1}, X_i\otimes_{G_i}U_i)$ and
  $\phi_i \in \Hom_{G_i}(U_i,X_i^*\otimes_{G_i} U_{i-1})$ for every
  $i\in \{1,\ldots,n\}$.
\end{itemize}
Denote by $\lambda^{(n)}$ the isomorphism obtained upon
applying $-\otimes_{G_0}U_0$ to (\ref{eq:15}),
\[
\lambda^{(n)}\colon X_n^* \otimes_{G_{n-1}} \cdots \otimes_{G_1} X_1^*
\otimes_{G_0}U_0
\xrightarrow{\sim}
(X_1 \otimes_{G_1}\cdots \otimes_{G_{n-1}} X_n)^*
\otimes_{G_0} U_0\,.
\]
Then,
\begin{equation}
  \label{eq:4}
  \langle
  f_n \times \cdots \times f_1 |
  \lambda^{(n)} \circ (\phi_1\times \cdots \times \phi_n)
  \rangle =
  \prod_{t=1}^n \langle f_i | \phi_i\rangle\,.
\end{equation}

\section{Intertwiners }
\label{sec:subsp-mund-i}

Let $Q_0$, $S$, $G$ and $M$ be as in
Setting~\ref{sec:defin-notat-main-2}. This section establishes
properties on intertwiners relative to $M$ which are needed to prove
Theorem~\ref{sec:defin-main-results}. This includes the properties of
$\circledast$ mentioned in Section~\ref{sec:defin-notat-main} as well
as Proposition~\ref{sec:defin-notat-main-13} on the existence of the
pairing $(-|-)$. This is done in several steps,
\begin{itemize}
\item first, by proving that, in
  Definition~\ref{sec:defin-notat-main-1}, the ${\mathbbm{k}} G_{i_0}$-module
  $M(\underline i;V)$ is isomorphic to
  $M(\underline i)\otimes_{G_{i_n}} V$ for a suitable
  ${\mathbbm{k}} G_{i_0}-{\mathbbm{k}} G_{i_n}$-bimodule $M(\underline i)$;
\item next, by proving that, under this isomorphism, $\circledast$
  corresponds to the operation $\times$ of
  Definition~\ref{sec:pairing-morphisms-1};
\item and finally by defining $(-|-)$ so that, under this isomorphism,
  it corresponds to the pairing $\langle-|-\rangle$ of
  Definition~\ref{sec:pairing-morphisms}.
\end{itemize}
These steps are proceeded in turn in Sections
\ref{sec:spaces-munderline-i-4}, \ref{sec:comp-intertw-1}, and
\ref{sec:pair-intertw}, respectively. They use the bimodule $MG$
defined as follows. Note that the action of $G$ on $S$ yields the
skew group algebra $S*G$. Like in Section~\ref{sec:pair-betw-morph},
for all groups $H$, the pieces of notation $e$, $\otimes_H$ and
$\Hom_H$ stand for the neutral element of $H$, for $\otimes_{{\mathbbm{k}} H}$,
and for $\Hom_{{\mathbbm{k}} H}$, respectively.

\begin{defn}
  \label{sec:intertwiners-}
  Define $MG$ as the following $S*G-S*G$-bimodule.
  \begin{itemize}
  \item   Its underlying vector space is
    $M\otimes_{\mathbbm{k}}{\mathbbm k} G$, a tensor
  $m\otimes g$ is denoted by $m*g$ for all $m\in M$ and $g\in G$.
\item The actions of $S$ and $G$ on the left and on the right are such
  that, for all $m\in M$, $g,h,k\in G$ and $s,s'\in S$,
\[
\begin{array}{ccc}
  s\cdot (m*g)\cdot s' & = & sm\,^gs'*g \\
  h \cdot (m*g) \cdot k & = & \,^hm *hgk\,.
\end{array}
\]
\end{itemize}
\end{defn}

The $S*G-S*G$-bimodule $M^*G$ is defined similarly after replacing $M$ by
$M^*$.

\subsection{The space $M(\underline{i})$}
\label{sec:spaces-munderline-i-4}

Let $n$ be a positive integer. Let $\underline{i}=i_0,\ldots,i_n$ a
sequence in $Q_0$.  Denote by $M(\underline{i})$ the following
${\mathbbm{k}} G_{i_0}-{\mathbbm{k}} G_{i_n}$-bimodule,
\begin{equation}
  \label{eq:24}
M(\underline{i}) = e_{i_0}MGe_{i_1}\otimes_{G_{i_1}}\cdots
\otimes_{G_{i_{n-1}}} e_{i_{n-1}} MG e_{i_n}\,.
\end{equation}
Note that this yields the ${\mathbbm{k}} G_{i_n}-{\mathbbm{k}} G_{i_0}$-bimodule
$M^*(\underline{i}^o)$. These bimodules are dual to each other. Here
is an explicit isomorphism used later. Consider the isomorphism
(\ref{eq:15}) taking $X_t=e_{i_{t-1}}M^*G e_{i_t}$ for all
$t\in \{1,\ldots,n\}$,
\begin{equation}
  \label{eq:10}
  e_{i_n}(MG)^*e_{i_{n-1}}\otimes_{G_{i_{n-1}}} \cdots \otimes_{G_{i_1}}
  e_{i_1}(MG)^* e_{i_0} \xrightarrow{(\ref{eq:15})}
  M(\underline{i})^*\,.
\end{equation}
Since the following mapping is a non-degenerate pairing 
\[
\begin{array}{crcl}
  [-|-] \colon & MG \otimes_{\mathbbm{k}} M^*G & \to & {\mathbbm{k}} \\
               & (m * g) \otimes (\varphi * h)
                                 & \mapsto  & \varphi(\,^{g^{-1}}m) \delta_{e,hg}\,,
\end{array}
\]
it induces an isomorphism (of $S*G-S*G$-bimodules)
\begin{equation}
  \label{eq:21}
  \begin{array}{rcl}
    M^*G & \to & (MG)^* \\
    \bullet & \mapsto & [-|\bullet]
  \end{array}
\end{equation}
the $n$-th tensor power of which induces an isomorphism of ${\mathbbm{k}}
G_{i_n}-{\mathbbm{k}} G_{i_0}$-bimodules,
\begin{equation}
  \label{eq:5}
  M^*(\underline{i}^o) \xrightarrow{\sim}
  e_{i_n}(MG)^*e_{i_{n-1}}\otimes_{G_{i_{n-1}}} \cdots \otimes_{G_{i_1}}
  e_{i_1}(MG)^* e_{i_0}\,.
\end{equation}
Composing (\ref{eq:10}) and (\ref{eq:5}) yields an
isomorphism of ${\mathbbm{k}} G_{i_n}-{\mathbbm{k}} G_{i_0}$-bimodules
\begin{equation}
  \label{eq:6}
  \Lambda_{\underline{i}}\colon M^*(\underline{i}^o) \xrightarrow{\sim}
  M(\underline{i})^*\,,
\end{equation}
which is induced by the mapping from $(MG)^{\otimes_{\mathbbm{k}} n}$ to
$((MG)^{\otimes_{\mathbbm{k}} n})^*$ which assigns to any tensor
$(\varphi_n * g_n) \otimes \cdots \otimes (\varphi_1 * g_1)$ the
following linear form on $(MG)^{\otimes_{\mathbbm{k}} n}$,
\begin{equation}
  \label{eq:35}
  (m_1 * g_1') \otimes \cdots \otimes (m_n * g_n') \mapsto \sum_{k_1,\ldots,k_{n-1}}\prod_{t=1}^n\varphi_t(\,^{k_t^{-1}g_t'^{-1}}m_t)\delta_{e,g_tk_{t-1}^{-1}g_t'k_t}\,,
\end{equation}
where $\delta$ denotes the Kronecker symbol, $k_t$ runs through
$G_{i_t}$ and $k_n=k_0=e$.

The bimodules $M(\underline i)$ and $M^*(\underline i^o)$ are
hence objects of the monoidal category $(\mathrm{mod}(A^e),\otimes_A)$ in
which they are left dual to each other. The following lemma describes
tensor products of these objects in this monoidal category. The same
statements are valid after replacing $M$ by $M^*$.
\begin{lem}
  \label{sec:spaces-munderline-i-3}
  Let $m,n$ be positive integers. Let $\underline{i''} =
  i_0,\ldots,i_{m+n}$ be a sequence in $Q_0$. Denote
  the sequences  $i_0,\ldots, i_m$ and $i_m,\ldots,i_{m+n}$ by
  $\underline{i}$ and $\underline{i'}$, respectively. 
  \begin{enumerate}
  \item $M(\underline{i''}) = M(\underline{i})\otimes_{G_{i_m}}
    M(\underline{i'})$.
  \item The following diagram is commutative
    \[
    \xymatrix{
      M^*(\underline{i'}^o) \otimes_{G_{i_m}} M^*(\underline{i}^o)
      \ar[rr]^{\Lambda_{\underline{i'}} \otimes_{G_{i_m}}
        \Lambda_{\underline{i}}}
      \ar@{=}[d]
      &&
      M(\underline{i'})^* \otimes_{G_{i_m}} M(\underline{i})^*
      \ar[d]^{(\ref{eq:9})} \\
      M^*(\underline{i''}^o) \ar[rr]_{\Lambda_{\underline{i''}}}
      &&
      M(\underline{i''})^* \, .
    }
    \]
  \end{enumerate}
\end{lem}
\begin{proof}
  (1) follows from the definition of the spaces
  $M(\underline{\bullet})$ and (2) follows from the definition of the
  morphisms $\Lambda_{\underline{\bullet}}$ in terms of (\ref{eq:15})
  and (\ref{eq:21}).
\end{proof}

The following result is the link between the bimodules
$M(\underline \bullet)$ and the modules of
Definition~\ref{sec:defin-notat-main-1}.  This link makes it possible
to translate the results of Section~\ref{sec:pair-betw-morph} in terms
of intertwiners. From now on, the piece of notation $*$ used for
tensors in $MG$ is often dropped in order to avoid overloaded
notation.

\begin{lem}
  \label{sec:space-munderlinei}
  Let $\underline i = i_0,\ldots,i_n$ be a sequence in $Q_0$, where
  $n\geqslant 1$. Let $U_{n}\in \mathrm{mod}({\mathbbm{k}} G_{i_n})$. There is an isomorphism
  in $\mathrm{mod}({\mathbbm{k}} G_{i_0})$,
\begin{equation}
  \label{eq:subsp-mund-i-1}
  \Theta^M_{\underline{i};U_{n}} \colon M(\underline{i}; U_{n})
  \xrightarrow{\sim} M(\underline{i})
  \otimes_{G_{i_n}}U_{n}
\end{equation}
given by
\[
m_1\otimes \cdots \otimes m_n \otimes y_1\cdots y_n u \mapsto
m_1y_1 \otimes
\,^{y_1^{-1}}m_2y_2 \otimes \cdots \otimes
\,^{(y_1\cdots y_{n-1})^{-1}}m_n y_n \otimes u\,.
\]
When there is no risk of confusion, $\Theta^M_{\underline i;U_n}$ is
denoted by $\Theta_{\underline i; U_n}$.
\end{lem}
\begin{proof}
  This is straightforward to check.
\end{proof}

The following natural transformation of functors from
$\mathrm{mod}({\mathbbm{k}} G_{i_0})$ to $\mathrm{mod}({\mathbbm{k}})$, defined by
$\Theta_{\underline{i};U_{n}}$, is hence an isomorphism of functors,
\begin{equation}
  \label{eq:18}
  \begin{array}{ccc}
    \Hom_{G_{i_0}}(-,M(\underline{i};U_{n}))
    & \xrightarrow{\sim} &
                           \Hom_{G_{i_0}}(- ,M(\underline{i})\otimes_{G_{i_n}} U_{n}) \\
    -  & \longmapsto & \Theta_{\underline{i};U_{n}} \circ -\,.
  \end{array}
\end{equation}

\subsection{Compositions of intertwiners}
\label{sec:comp-intertw-1}

It is now possible to compare the operation $\circledast$ on
intertwiners, see Definition~\ref{sec:defin-notat-main-6}, and the
operation $\times$ of Definition~\ref{sec:pairing-morphisms-1}.

\begin{lem}
  \label{sec:comp-intertw}
  Let $\underline i'' = i_0,\ldots,i_{m+n}$ be a sequence in $Q_0$,
  where $m,n\geqslant 1$. Denote the sequences $i_0,\ldots,i_m$ and
  $i_m,i_{m+1},\ldots,i_{m+n}$ by $\underline i$ and $\underline i'$,
  respectively. Let $U_0\in \mathrm{mod}({\mathbbm{k}} G_{i_0})$,
  $U_m \in \mathrm{mod}({\mathbbm{k}} G_{i_m})$ and
  $U_{m+n}\in \mathrm{mod}({\mathbbm{k}} G_{i_{m+n}})$. Finally, let
  $f\in \Hom_{{\mathbbm{k}} G_{i_0}}(U_0,M(\underline i; U_m))$ and
  $f'\in \Hom_{{\mathbbm{k}} G_{i_m}}(U_m, M(\underline i';U_{m+n}))$.
  Then,
  \begin{equation}
    \label{eq:7}
  (\Theta_{\underline{i'};U_{m+n}} \circ f') \times
  (\Theta_{\underline{i};U_m} \circ f) =
  \Theta_{\underline{i''};U_{m+n}} \circ (f' \circledast f)\,.
\end{equation}
  As a consequence,
  $f'\circledast f \in \Hom_{G_{i_0}}( U_0 ,
  M(\underline{i''};U_{m+n}))$.
\end{lem}
\begin{proof}
  Let $u\in
  U_0$. Since $\Theta_{\underline{i'};U_{m+n}} \circ f'$ lies in
  $\Hom_{G_{i_m}}(U_m, M(\underline{i''}) \otimes_{G_{i_{m+n}}}
  U_{m+n})$ and $\Theta_{\underline{i};U_m} \circ f$ lies in $\Hom_{G_{i_0}}(
  U_0, M(\underline{i}) \otimes_{G_{i_m}} U_m)$, then
\[
  ((\Theta_{\underline{i'};U_{m+n}} \circ f') \times
  (\Theta_{\underline{i}; U_m} \circ f))(u)
  \underset{(\ref{eq:3})}{=} 
  (M(\underline{i}) \otimes_{{\mathbbm{k}} G_{i_m}} (\Theta_{\underline{i'};U_{m+n}}
  \circ f')) \circ (\Theta_{\underline{i}; U_m}\circ f)(u)\,.
\]
Using the definition of $\Theta_{\underline i;U_m}$, see
(\ref{eq:subsp-mund-i-1}), the right-hand side is equal to
\[
\begin{array}{l}
  (M(\underline{i}) \otimes_{{\mathbbm{k}} G_{i_m}} (\Theta_{{\underline{i}'}; U_{m+n}} \circ f'))
  (
  \sum_{\underline{y}}
  f_{\underline{y}}^{(1)}(u) y_1 \otimes
  \,^{y_1^{-1}} f_{\underline{y}}^{(2)}(u) y_2 \otimes
  \cdots  \\

  \cdots \otimes
  \,^{(y_1\cdots y_{m-1})^{-1}} (f_{\underline{y}}^{(m)}(u)) y_m \otimes
  f_{\underline{y}}^{(0)}(u) )\,,
\end{array}
\]
where $\underline y$ runs through all sequences $y_1,\ldots,y_m$ such
that $y_t \in [G/G_{i_t}]$ for all $t\in \{1,\ldots,m\}$. Using the
definition of $\Theta_{\underline{i}';U_{m+n}}$, see
(\ref{eq:subsp-mund-i-1}), this is equal to
\[
\begin{array}{l}
  \sum_{\underline{y},\underline{y'}}
  f_{\underline{y}}^{(1)}(u) y_1 \otimes \cdots \otimes
  \,^{(y_1\cdots y_{m-1})^{-1}} (f_{\underline{y}}^{(m)}(u)) y_m \otimes \\
  
  f_{\underline{y'}}'^{(1)}(f_{\underline{y}}^{(0)}(u)) y_{m+1} \otimes
  \,^{y_{m+1}^{-1}} (f_{\underline{y'}}'^{(2)}(f_{\underline{y}}^{(0)}(u))) y_{m+2} \otimes \cdots \\
  
  \cdots \otimes   \,^{(y_{m+1}\cdots y_{m+n})^{-1}}
  (f_{\underline{y'}}'^{(n)}(f_{\underline{y}}^{(0)}(u))) y_{m+n} \otimes
  f_{\underline{y'}}'^{(0)}(f_{\underline{y}}^{(0)}(u))\,,
\end{array}
\]
where $\underline y$ is as before and $\underline{y'}$ runs through
all sequences $y_{m+1},\ldots,y_{m+n}$ such that
$y_{m+t} \in [G/G_{i_{m+t}}]$ for all $t\in \{1,\ldots,n\}$.  Using
the definition of $\Theta_{\underline{i}'';U_{m+n}}$, see
(\ref{eq:subsp-mund-i-1}), this is equal to
\[
\begin{array}{l}
  \Theta_{\underline{i''};U_{m+n}}(
  \sum_{\underline{y''}} f_{\underline{y}}^{(1)}(u) \otimes \cdots
  \otimes
  f_{\underline{y}}^{(m)}(u) \otimes
  \, ^{y_1\cdots y_m} (f_{\underline{y'}}'^{(1)}(f_{\underline{y}}^{(0)}(u))) \otimes \cdots  \\

  \cdots \otimes
  \, ^{y_1\cdots y_m}(f_{\underline{y'}}'^{(n)}( f_{\underline{y}}^{(0)}(u))) \otimes
  y_1\cdots y_{m+n}
  f_{\underline{y'}}'^{(0)}(f_{\underline{y}}^{(0)}(u)))\,,
\end{array}
\]
where $\underline{y''}$ runs through all sequences
$y_1,\ldots,y_{m+n}$ such that $y_t\in [G/G_{i_t}]$ for all
$t\in \{1,\ldots,m+n\}$ and where $\underline{y}$ and $\underline{y'}$
stand for the sequences $y_1,\ldots,y_m$ and $y_{m+1},\ldots,y_{m+n}$,
respectively.  Finally, using the definition of $\circledast$, see
(\ref{eq:36}), this is equal to
$(\Theta_{\underline{i''};U_{m+n}}\circ (f' \circledast f))(u)$.
Whence (\ref{eq:7}). Note that $f'\circledast f$ is a morphism of
${\mathbbm{k}} G_{i_0}$-modules because
\begin{itemize}
\item $\Theta_{\underline{i''};U_{m+n}}$ is an isomorphism of ${\mathbbm{k}}
  G_{i_0}$-modules,
\item $f$ and $\Theta_{\underline{i};U_m}$ are morphisms of ${\mathbbm{k}}
  G_{i_0}$-modules,
\item $f'$ and $\Theta_{\underline{i'};U_{m+n}}$ are morphisms of ${\mathbbm{k}}
  G_{i_m}$-modules,
\item and the operation ``$\times$'' yields morphisms of
  ${\mathbbm{k}} G_{i_0}$-modules (see (\ref{eq:3})).
\end{itemize}
\end{proof}

Lemma~\ref{sec:comp-intertw} shows that
Definition~\ref{sec:defin-main-results-3} yields a, possibly
non-associative, algebra.
\begin{prop}
  \label{sec:comp-intertw-2}
  $\intw$ is an associative and unital algebra.
\end{prop}
\begin{proof}
  Note that $\intw$ is unital. Indeed, for all
  $i\in [G\backslash Q_0]$ and $U\in \irr(G_i)$, denote by ${\varepsilon}_{i,U}$
  the intertwiner $\id_U \in \Hom_{{\mathbbm{k}} G_i}(U, M(i;U))$, recall that
  $M(i;U)=U$. Then $\sum_{i,U}{\varepsilon}_{i,U}$ is a unity for $\circledast$.

  In the framework of Definition~\ref{sec:pairing-morphisms-1}, the
  equality $f_3\times(f_2\times f_1) = (f_3\times f_2)\times f_1$
  holds whenever both sides make sense because of the associativity of
  the tensor product. In view of Lemma~\ref{sec:comp-intertw}, this
  entails that $\circledast$ is associative.
\end{proof}

\subsection{The pairing on intertwiners}
\label{sec:pair-intertw}

Let $n$ be a positive integer, let $\underline{i}=i_0,\ldots,i_n$ be a
sequence in $Q_0$ and let $U_0$ and $U_n$ be left modules over
${\mathbbm{k}} G_{i_0}$ and ${\mathbbm{k}} G_{i_n}$, respectively.  To every
$\varphi \in \Hom_{G_{i_n}}(U_n,M^*(\underline{i}^o;U_0))$ is
associated a ${\mathbbm{k}}$-linear mapping
\[
\Phi(\varphi) \colon M(\underline{i}; U_n) \to U_0
\]
defined as follows. For every sequence of representatives $y_1,\ldots,y_n$
($y_t\in \left[G/G_{i_t}\right]$) there exist unique sequences
\[
\underline{x}=x_{n-1},\ldots,x_0\ (x_t \in \left[G/G_{i_t}\right])
\,,\ \text{and}\
h_{n-1},\ldots,h_0\ (h_t\in G_{i_t})
\]
such that, for all $t\in \{1,\ldots,n\}$,
\begin{equation}
  \label{eq:41}
(y_t\cdots y_n)^{-1}  = x_{n-1}\cdots x_{t-1} h_{t-1}\,;
\end{equation}
these sequences are determined by a decreasing induction:
$y_n^{-1} = x_{n-1}h_{n-1}$ and $h_ty_t^{-1} = x_{t-1}h_{t-1}$ for all
$t\in \{1,\ldots,n-1\}$; then, for every
$m_1\otimes \cdots \otimes m_t \otimes y_1\cdots y_n u$ lying in
$M(\underline{i};U_n)$, define the element
$\Phi(\varphi)(m_1\otimes \cdots \otimes m_t \otimes y_1\cdots y_n u)$
of $U_0$, using Notation~\ref{sec:defin-notat-main-12}, as
\begin{equation}
  \label{eq:30}
\prod_{t=1}^n \varphi^{\underline{x}}_{(t)}(u)\left(\,^{(y_1\cdots y_n)^{-1}}m_t
\right) \,^{h_0^{-1}} (\varphi^{\underline{x}}_{(0)}(u))\,.
\end{equation}
This does make sense, because
$m_t \in \,_{y_1\cdots y_{t-1}\cdot i_{t-1}}M_{y_1 \cdots y_t \cdot
  i_t}$,
because $\varphi^{\underline{x}}_{(t)}$ is a linear form on
$_{x_{n-1}\cdots x_{t-1}\cdot i_{t-1}}M_{x_{n-1}\cdots x_t\cdot i_t}$,
and because of (\ref{eq:41}).

After appropriate identifications, $\Phi$ is related to the adjunction
isomorphism $\Psi_{M(\underline{i})}^{\bullet, \bullet}$ of
section~\ref{sec:adjunctions} as explained in the following
result. Recall that $\Psi_{M(\underline{i})}^{U_n,U_0}$,
$\Theta^{M^*}_{\underline{i}^o;U_0}$ and $\Lambda_{\underline{i}}$ are defined
in (\ref{eq:16}), (\ref{eq:subsp-mund-i-1}), and (\ref{eq:6}), respectively.
\begin{lem}
  \label{sec:transl-intertw-1}
  Keep the previous setting. Denote by $\varphi'$ the following
  composite morphism of left ${\mathbbm{k}} G_{i_n}$-modules,
  \[
  U_n \xrightarrow{\varphi} M^*(\underline{i}^o;U_0)
  \xrightarrow{\Theta^{M^*}_{\underline{i}^o;U_0}}
  M^*(\underline{i}^o)\otimes_{G_{i_0}} U_0
  \xrightarrow{\Lambda_{\underline{i}} \otimes_{G_{i_0}} U_0}
  M(\underline{i})^* \otimes_{G_{i_0}} U_0\,.
  \]
  Then,
  \begin{equation}
    \label{eq:19}
    \Phi (\varphi)=\Psi_{M(\underline{i})}^{U_n,U_0} (\varphi')\circ
    \Theta^M_{\underline{i};U_n}\,.
  \end{equation}
  In particular,
  $\Phi(\varphi) \in \Hom_{G_{i_0}}(M(\underline{i};U_n), U_0)$.
\end{lem}
\begin{proof}
  The composition $\Theta^{M^*}_{\underline{i}^o;U_0}\circ \varphi$ is given
  by (see (\ref{eq:subsp-mund-i-1}))
  \[
  u \mapsto
  \sum_{\underline{x}} \varphi^{\underline{x}}_{(n)}(u) x_{n-1} \otimes
  \,^{x_{n-1}^{-1}} \varphi^{\underline{x}}_{(n-1)}(u) x_{n-2} \otimes
  \cdots \otimes \,^{(x_{n-1}\cdots x_1)^{-1}} \varphi^{\underline{x}}_{(1)}(u) x_0 \otimes \varphi^{\underline{x}}_{(0)}(u)\,,
  \]
  where $\underline{x} = x_{n-1},\ldots,x_0$ runs through all sequences
  of representatives ($x_t\in [G/G_{i_t}]$ for all
  $0 \leqslant t \leqslant n-1$). Hence, following (\ref{eq:35}), the
  morphism $\varphi'$ is given by
  \[
  u \mapsto \sum_{\underline{x}} \phi^{\underline{x}}(u) \otimes
  \varphi^{\underline{x}}_{(0)}(u)\,
  \]
  where $\phi^{\underline{x}}(u)$ is the linear form on
  $M(\underline{i})$, which is a quotient of $(MG)^{\otimes_{\mathbbm{k}} n}$,
  induced by the linear form on $(MG)^{\otimes_{\mathbbm{k}} n}$ mapping any
  $(m_1 * g_1) \otimes \cdots \otimes (m_n * g_n)$ to
  \begin{equation}
    \label{eq:38}
  \sum_{k_1,\ldots,k_{n-1}} \prod_{t=1}^n
  \left(
    \,^{(x_{n-1}x_{n-2}\cdots x_t)^{-1}} (\varphi^{\underline{x}}_{(t)}(u))
  \right)
  \left(
    \,^{k_t^{-1}g_t^{-1}} m_t
    \right)
    \cdot
    \delta_{e,x_{t-1}k_{t-1}^{-1}g_tk_t}\,,
  \end{equation}
  where $\delta$ denotes the Kronecker symbol, $k_0=k_n=e$ and $k_t$
  runs through $G_{i_t}$ for $1\leqslant t \leqslant n-1$.

  In order to simplify the expression of (\ref{eq:38}), consider
  a tensor $(m_1 * g_1) \otimes \cdots \otimes (m_n * g_n)$ in
  $(MG)^{\otimes_{\mathbbm{k}} n}$; for all $k_1,\ldots,k_{n-1}$ in $G$, the
  following assertions are equivalent, keeping $k_0=k_n=e$,
  \begin{enumerate}[(i)]
  \item $(\forall t\in \{1,\ldots,n-1\})\ e = x_{t-1}k_{t-1}^{-1} g_t
    k_t$,
  \item $k_t = (x_{t-1}\cdots x_1 x_0 g_1 g_2\cdots g_t)^{-1}$ for all
    $t\in \{1,\ldots,n\}$, and $g_1\cdots g_n$ is the inverse of
    $x_{n-1}\cdots x_1 x_0$;
  \end{enumerate}
  given that each
  $^{(x_{n-1}x_{n-2}\cdots x_t)^{-1}}( \varphi^{\underline{x}}_{(t)}(u))$
  in (\ref{eq:38}) is a linear form on $_{x_{t-1}i_{t-1}}M_{i_t}$ and
  each $k_t$ is constrained to lie in $G_{i_t}$, then (\ref{eq:38}) is
  equal to
  \[
  \prod_{t=1}^n (\varphi^{\underline{x}}_{(t)}(u)) \left(
    \,^{x_{n-1}\cdots x_1x_0g_1g_2\cdots g_{t-1}} m_t
  \right)
  \delta_{e,g_1g_2\cdots g_n x_{n-1}\cdots x_1x_0}\,,
  \]
  which simplifies to
  \begin{equation}
    \label{eq:39}
    \prod_{t=1}^n (\varphi^{\underline{x}}_{(t)}(u))
    \left(
      \,^{(g_tg_{t+1}\cdots g_n)^{-1}} m_t
    \right)
    \delta_{(g_1g_2\cdots g_n)^{-1},x_{n-1}\cdots x_1x_0}\,.
  \end{equation}

  Thus, the morphism $\varphi'$ is given by
  $u \mapsto \sum_{\underline{x}}\phi^{\underline{x}}(u) \otimes
  \varphi^{\underline{x}}_{(0)}(u)$,
  where $\phi^{\underline{x}}(u)$ is the linear form on
  $M(\underline{i})$ induced by the mapping assigning (\ref{eq:39}) to
  any element $(m_1 * g_1) \otimes \cdots \otimes (m_n * g_n)$ of
  $(MG)^{\otimes_{\mathbbm{k}} n}$.

  Hence (see Lemma~\ref{sec:adjunctions}),
  $\Psi_{M(\underline{i})}^{U_n,U_0}(\varphi')$ is the morphism from
  $M(\underline{i})\otimes_{G_{i_n}} U_n$ to $U_0$ induced by the
  mapping from $(MG)^{\otimes_{\mathbbm{k}} n}\otimes_{\mathbbm{k}} U_n$ to $U_0$ which maps any
  $(m_1*g_1) \otimes \cdots \otimes (m_n * g_n) \otimes u$ to
  \[
  \begin{array}{ll}
    &
      \sum_{\underline{x}}\sum_{g\in G_{i_0}}
      \phi^{\underline{x}}(u)(g^{-1}m_1g_1\otimes m_2g_2\otimes \cdots
      \otimes m_n g_n)\ ^g(\varphi^{\underline{x}}_{(0)}(u)) \\
    =&
       \sum_{g\,,\, \underline{x}}
       \prod_{t=1}^n\varphi^{\underline{x}}_{(t)}(u)(\,^{(g_t\cdots
       g_n)^{-1}}m_t)
       \delta_{(g_1\cdots g_n)^{-1}g,x_{n-1}\cdots
       x_0}\ ^g(\varphi^{\underline{x}}_{(0)}(u))\,.
  \end{array}
  \]

  Now, let
  $m_1 \otimes \cdots \otimes m_n \otimes y_1\cdots y_n u\in
  M(\underline{i};U_n)$.
  Recall that its image under $\Theta^M_{\underline{i}; U_n}$ is
  $m_1 y_1 \otimes \,^{y_1^{-1}}m_2y_2 \otimes \cdots \otimes
  \,^{(y_1\cdots y_{n-1})^{-1}} m_n y_n \otimes u$
  (see (\ref{eq:subsp-mund-i-1})). Hence, its image under
  $\Psi_{M(\underline{i})}^{U_n,U_0}(\varphi') \circ \Theta^M_{\underline{i};U_n}$ is
  \begin{equation}
    \label{eq:14}
    \sum_{g,\underline{x}} \prod_{t=1}^n \varphi^{\underline{x}}_{(t)}(u)
    (\,^{(y_1\cdots y_n)^{-1}}m_t)
    \delta_{(y_1 \cdots y_n)^{-1}g,x_{n-1}\cdots
      x_0}\ ^g(\varphi^{\underline{x}}_{(0)}(u))\,.
  \end{equation}
  Note that, given $g\in G_{i_0}$ and $\underline{x}$, then, for all
  $t\in \{1,\ldots,n\}$,
  \begin{itemize}
  \item $\varphi^{\underline{x}}_{(t)}(u) \in (\,_{x_{n-1}\cdots
      x_{t-1}\cdot i_{t-1}} M _{x_{n-1}\cdots x_t\cdot i_t})^*$ and
  \item $^{(y_1\cdots y_n)^{-1}}m_t \in \,_{(y_t\cdots y_n)^{-1}\cdot
      i_{t-1}} M _{(y_{t+1}\cdots y_n)^{-1}\cdot i_t}$;
  \end{itemize}
  in particular, if for a given index $(g,\underline{x})$ in
  (\ref{eq:14}) the product ``$\prod_{t=1}^n$'' is non-zero then this
  product is the unique non-zero term of the sum (\ref{eq:14}). More
  precisely, let $\{x_t,h_t\}_{0\leqslant t \leqslant n-1}$ be the
  unique collection of elements in $G$ such that, for all
  $t\in \{0,\ldots,n-1\}$,
  \begin{itemize}
  \item $x_t\in \left[G/G_{i_t}\right]$,
  \item $h_t\in G_{i_t}$ and
  \item $(y_{t+1}\cdots y_n)^{-1} = x_{n-1}\cdots x_t h_t$;
  \end{itemize}
  then, the only possibly non-zero term of the sum (\ref{eq:14}) is the
  one with index $(h_0^{-1},\underline{x})$. Therefore, (\ref{eq:14})
  equals
  \[
  \prod_{t=1}^n \varphi^{\underline{x}}_{(t)}(u)(\, ^{(y_1\cdots
    y_n)^{-1}}m_t)\ ^{h_0^{-1}}(\varphi^{\underline{x}}_{(0)}(u))\,,
  \]
  which is equal to
  $\Phi(\varphi)(m_1\otimes \cdots \otimes m_n\otimes y_1\cdots
  y_nu)$ (see (\ref{eq:30})).
  Thus,
  $\Phi(\varphi) = \Psi_{M(\underline{i})}^{U_n,U_0}(\varphi') \circ
  \Theta^M_{\underline{i};U_n}$.
  This proves (\ref{eq:19}).  The second assertion of the lemma is a
  direct consequence of (\ref{eq:19}).
\end{proof}

It follows from Lemma~\ref{sec:transl-intertw-1} that the following
mapping is bijective,
\[
\Phi \colon \Hom_{G_{i_n}}(U_n, M^*(\underline{i}^o;U_0)) \to
\Hom_{G_{i_0}}(M(\underline{i}; U_n), U_0)\,.
\]
This yields a non-degenerate pairing defined as follows.

\begin{defn}
  \label{sec:pair-intertw-1}
  Let $\underline i=i_0,\ldots,i_n$ be a sequence in $Q_0$, where
  $n\geqslant 1$. Let $U_0\in \mathrm{mod}({\mathbbm{k}} G_{i_0})$ be simple. Let
  $U_n\in \mathrm{mod}({\mathbbm{k}} G_{i_n})$. Define $(-|-)$ to be the non-degenerate
  pairing
\[
( -| - ) \colon \Hom_{G_{i_0}}(U_0, M(\underline{i}; U_n)) \otimes_{\mathbbm{k}}
\Hom_{G_{i_n}}(U_n, M^*(\underline{i}^o; U_0)) \to {\mathbbm{k}} \,,
\]
such that, for all $f\in \Hom_{G_{i_0}}(U_0, M(\underline{i}; U_n))$
and $\varphi \in \Hom_{G_{i_n}}(U_n, M^*(\underline{i}^o,U_0))$,
\begin{equation}
  \label{eq:13}
  ( f|\varphi ) \cdot \id_{U_0} = \Phi(\varphi) \circ
  f\,;
\end{equation}
\end{defn}

With this definition, it follows from (\ref{eq:34}) and (\ref{eq:19})
that
\begin{equation}
  \label{eq:20}
  (f | \varphi ) = \langle \Theta^M_{\underline{i}; U_n} \circ f |
  (\Lambda_{\underline{i}}\otimes_{G_{i_0}} U_0) \circ
  \Theta^{M^*}_{\underline{i}^o; U_0}
  \circ \varphi\rangle\,.
\end{equation}

It is now possible to prove Proposition~\ref{sec:defin-notat-main-13}.
\begin{proof}[Proof of Proposition~\ref{sec:defin-notat-main-13}]
  Note that~(\ref{eq:29}) is equal to $(\Phi(\varphi))(f(u))$, see
  (\ref{eq:30}). Hence, the pairing $(-|-)$ claimed by the proposition
  is the one of Definition~\ref{sec:pair-intertw-1}.
\end{proof}

The pairing $(-|-)$ is compatible with the composition of (dual)
intertwiners in the following sense.
\begin{lem}
  \label{sec:pair-intertw-2}
  Let $m,n$ be positive integers, let
  $\underline{i''}=i_0,\ldots,i_{m+n}$ be a sequence in
  $Q_0$, let $U_0$ and $U_m$ be simple left modules over ${\mathbbm{k}} G_{i_0}$
  and ${\mathbbm{k}} G_{i_m}$, respectively, and let $U_{m+n}$ be a left
  ${\mathbbm{k}} G_{i_{m+n}}$-module.  Then, for all intertwiners
  \[
  \begin{array}{ll}
    f_1\colon U_0 \to M(\underline{i};U_m)\,,& \varphi_1 \colon U_m \to
                                              M^*(\underline{i}^o;U_0)\,,
    \\
    f_2\colon U_m \to M(\underline{i'};U_{m+n})\,,& \varphi_2 \colon U_{m+n} \to
                                                    M^*(\underline{i'}^o;U_m)\,,
  \end{array}
  \]
  where $\underline{i}=i_0,\ldots,i_m$ and $\underline{i'}=i_m,\ldots,i_{m+n}$,
  the following holds
  \begin{equation}
    \label{eq:22}
  (f_2 \circledast f_1 | \varphi_1 \circledast \varphi_2 ) = ( f_2 |\varphi_2
  ) \cdot ( f_1 | \varphi_1 )\,.
\end{equation}
\end{lem}
\begin{proof}
  In order to avoid overloaded expressions, the following shorthand
  notation is used in this proof: $\Theta_{\underline i;U_0}$ for
  $\Theta^M_{\underline i; U_0}$ and $\Theta'_{\underline{i};U_0}$ for
  $\Theta^{M^*}_{\underline{i}^o; U_0}$. Define
  $f_1',f_2',\varphi_1',\varphi_2'$ as follows
  \[
  \begin{array}{ll}
    f_1' = \Theta_{\underline{i};U_m} \circ f_1\,,
    &
      \varphi_1' = (\Lambda_{\underline{i}} \otimes_{G_{i_0}} U_0) \circ \Theta'_{\underline{i};U_0} \circ \varphi_1\,, \\
    f_2' = \Theta_{\underline{i'};U_{m+n}} \circ f_2\,,
    &
      \varphi_2' = (\Lambda_{\underline{i'}} \otimes_{G_{i_m}} U_m) \circ
      \Theta'_{\underline{i'};U_m} \circ \varphi_2\,.
  \end{array}
  \]
  % In particular
  % \[
  % \left\{
  %   \begin{array}{rcl}
  %     \Phi(\varphi_1) & = & \Psi(\varphi_1') \circ \Theta_{\underline
  %                           i;U_m} \\
  %     \Phi(\varphi_2) & = & \Psi(\varphi_2') \circ
  %                           \Theta_{\underline{i'};U_{m+1}}\,.
  %   \end{array}
  % \right.
  % \]
  % Consider the following composite morphism equal to
  % $(f_2\circledast f_1|\varphi_1\circledast\varphi_2)\cdot
  % \Theta_{\underline{i''};U_{m+n}}$,
  % \[
  % \Theta_{\underline{i''};U_{m+n}}\circ (f_2 \circledast f_1) \circ
  % \Phi(\varphi_1 \circledast \varphi_2)\,.
  %   \]
  % It follows from Lemma~\ref{sec:spaces-munderline-i-3},
  % Lemma~\ref{sec:comp-intertw} and Lemma~\ref{sec:transl-intertw-1},
  % that this morphism equals the following one
  % \[
  % (f_2' \times f_1') \circ \Psi(
  % \lambda \circ
  % (\varphi_1' \times \varphi_2')) \circ
  % \Theta_{\underline{i''};U_{m+n}}\,,
  % \]
  % where
  % $\lambda \colon M(\underline{i'})^* \otimes_{G_{i_m}} M(\underline
  % i)^* \otimes_{G_{i_0}} U_0 \to M(\underline{i''})^*
  % \otimes_{G_{i_0}} U_0$ is induced by (\ref{eq:9}). And it follows
  % from Lemma~\ref{sec:pair-betw-morph-4} that this morphism equals
  % $\langle f_2'|\varphi_2'\rangle \cdot \langle f_1'
  % |\varphi_1'\rangle\cdot  \Theta_{\underline{i''};U_{m+n}}$. Now,
  % $(f_1|\varphi_1) = \langle f_1' | \varphi_1'\rangle$ and
  % $(f_2|\varphi_2) = \langle f_2' | \varphi_2'\rangle$ (see (\ref{eq:20})). Since
  % $\Theta_{\underline{i''};U_{m+n}}$ is an isomorphism, it follows
  % that $(f_2\circledast f_1|\varphi_1\circledast\varphi_2) =
  % (f_1|\varphi_1)\cdot (f_2|\varphi_2)$.
  It follows from (\ref{eq:20}) that
  $(f_2 \circledast f_1 | \varphi_1 \circledast \varphi_2)$ is equal
  to
  \[
  \langle \Theta_{\underline{i''},U_{m+n}} \circ (f_2 \circledast f_1) |
  (\Lambda_{\underline{i''}} \otimes_{G_{i_0}} U_0) \circ
  \Theta'_{\underline{i''},U_0} \circ (\varphi_1 \circledast \varphi_2)\rangle\,;
  \]
  applying Lemma \ref{sec:comp-intertw} twice, first for $M$ and next
  for $M^*$, yields that the previous expression is equal to
  \[
  \langle
  \underset{f_1'}{\underbrace{(\Theta_{\underline{i'},U_{m+n}} \circ f_2)}} \times
  \underset{f_2'}{\underbrace{(\Theta_{\underline{i},U_m} \circ f_1)}}
  |
  (\Lambda_{\underline{i''}}\otimes_{G_{i_0}} U_0) \circ
  \left[
    (\Theta'_{\underline{i},U_0} \circ \varphi_1) \times
    (\Theta'_{\underline{i'},U_m} \circ \varphi_2)
    \right]
  \rangle\,;
  \]
  according to Lemma~\ref{sec:spaces-munderline-i-3} this is equal to
  \begin{equation}
    \label{eq:40}
  \langle
  f_1' \times f_2'
  |
  \lambda \circ
  (\Lambda_{\underline{i'}} \otimes_{G_{i_m}} \Lambda_{\underline{i}}
  \otimes_{G_{i_0}} U_0) \circ
  \left[
    (\Theta'_{\underline{i},U_0} \circ \varphi_1) \times
    (\Theta'_{\underline{i'},U_m} \circ \varphi_2)
    \right]
  \rangle\,,
\end{equation}
where $\lambda$ is as in Lemma~\ref{sec:pair-betw-morph-4}.  Now,
since the following composite morphisms are equal due to the
associativity of the tensor product
\[
\def\objectstyle{\scriptscriptstyle}
\def\labelstyle{\scriptscriptstyle}
\xymatrix{
  M^*(\underline{i'}^o) \underset{{G_{i_m}}}\otimes U_m
  \ar[rr]^{
    1 \otimes (\Theta'_{\underline{i}; U_0} \circ \varphi_1)
  }
  &&
  M^*(\underline{i'}^o) \underset{G_{i_m}}\otimes M^*(\underline{i}^o)
  \underset{G_{i_0}}\otimes U_0
  \ar[rr]^{
    \Lambda_{\underline{i'}}\otimes \Lambda_{\underline{i}} \otimes 1
  }
  &&
  M(\underline{i'})^* \underset{G_{i_m}}\otimes M(\underline{i})^*
  \underset{G_{i_0}}\otimes U_0
  }
\]
and 
\[
\def\objectstyle{\scriptscriptstyle}
\def\labelstyle{\scriptscriptstyle}
\xymatrix{
  M^*(\underline{i'}^o) \underset{G_{i_m}}\otimes U_m
  \ar[r]^{
    \Lambda_{\underline{i'}} \otimes 1}
&
M(\underline{i'})^* \underset{G_{i_m}}\otimes U_m
\ar[r]^{
  1 \otimes
  (\Theta'_{\underline{i};U_0} \circ \varphi_1)
  }
&
M(\underline{i'})^* \underset{G_{i_m}}\otimes M^*(\underline{i}^o)
\underset{G_{i_0}}\otimes U_0
\ar[r]^{
  1 \otimes \Lambda_{\underline{i}} \otimes 1
}
&
M(\underline{i'})^* \underset{G_{i_m}}\otimes M(\underline{i})^*
\underset{G_{i_0}}\otimes U_0\,,
}
\]
then 
$(\Lambda_{\underline{i'}} \otimes_{G_{i_m}} \Lambda_{\underline{i}}
  \otimes_{G_{i_0}} U_0) \circ
  \left[
    (\Theta'_{\underline{i},U_0} \circ \varphi_1) \times
    (\Theta'_{\underline{i'},U_m} \circ \varphi_2)
  \right]$ is equal to
  \[
  [
    \underset{\varphi_1'}{
      \underbrace{
    (\Lambda_{\underline{i}} \otimes_{G_{i_0}} U_0) \circ
    \Theta'_{\underline{i};U_0} \circ \varphi_1
    }}
  ]
  \times
  [
    \underset{\varphi_2'}{
      \underbrace{
    (\Lambda_{\underline{i'}} \otimes_{G_{i_m}} U_m) \circ
    \Theta'_{\underline{i'};U_m} \circ \varphi_1
    }}
].
\]
Therefore,
\[
\begin{array}{rcl}
  (\ref{eq:40})
  &
    =
  &
    \langle f_2' \times f_1' | \lambda \circ
    (\varphi_1' \times \varphi_2')\rangle \\
  &
    \underset{\text{Lemma \ref{sec:pair-betw-morph-4}}}=
  &
   \langle f_2'|\varphi_2'\rangle \cdot \langle f_1'|\varphi_1'\rangle
  \\
  &
    \underset{(\ref{eq:20})}=
  &
    (f_2|\varphi_2)\cdot (f_1|\varphi_1)\,.
\end{array}
\]
Thus, $(f_2\circledast f_1|\varphi_1\circledast\varphi_2) =
 (f_1|\varphi_1)\cdot (f_2|\varphi_2)$.
\end{proof}

\section{Proof of the main results}
\label{sec:proof-main-results}

This section proves Theorems~\ref{sec:defin-main-results} and
\ref{sec:defin-main-results-4}. Keeping $Q_0$, $S$, $G$ and $M$ as in
Setting~\ref{sec:defin-notat-main-2} and $Q_G$ as in
Setting~\ref{sec:defin-notat-main-10}, the proofs use the following
notation:
\begin{itemize}
\item denote by $\hat e$ the idempotent $\sum_{i\in [G\backslash Q_0]}
  e_i$ of $S$, this idempotent is not to be confused with $\tilde e$
  which is equal to $\sum_{(i,U)}e_i*{\varepsilon}_U$, where $(i,U)$ runs through
  the vertices of $Q_G$;
\item denote by $S_1$ the ${\mathbbm{k}}$-algebra
  $\Pi_{i\in [G\backslash Q_0]} {\mathbbm{k}} G_i$.
\end{itemize}

First, it is necessary to check that the quiver $Q_G$ corresponds to
the one given by Demonet in \cite{MR2578593}.
\begin{proof}[Proof of Theorem~\ref{sec:proof-main-results-1}]
  Following \cite{MR2578593}, see p.~1057, the algebra $T_S(M)*G$ is
  Morita equivalent to $\tilde e \cdot(T_S(M)*G)\cdot \tilde e$, which
  is isomorphic to the path algebra of a quiver having the same
  vertices as $Q_G$ and such that the arrows from a vertex $(i,U)$ to
  a vertex $(j,V)$ form a basis of the vector space
  $(e_i*{\varepsilon}_U)\cdot (MG) \cdot (e_j*{\varepsilon}_V)$. Now, considering $e_i$,
  $e_j$, ${\varepsilon}_U$, and ${\varepsilon}_V$ as elements of $S*G$,
  \[
  \begin{array}{rcl}
    (e_i*{\varepsilon}_U)\cdot
    (MG) \cdot (e_j*{\varepsilon}_V) & = & {\varepsilon}_U\cdot(e_i\cdot (MG)\cdot e_j)\cdot
                                {\varepsilon}_V \\
                          & \underset{~(\ref{eq:24})}{=} & {\varepsilon}_U \cdot M(i,j) \cdot {\varepsilon}_V \\
                          & \simeq & \Hom_{{\mathbbm{k}} G_i}(U,M(i,j)\otimes_{G_j}V) \\
                          &
                            \underset{~(\ref{eq:subsp-mund-i-1})}{\simeq}
                              & \Hom_{{\mathbbm{k}} G_i}(U, M(i,j;V))\,.
  \end{array}
  \]
  Whence the first statement of the theorem. Now, if
  $U={\mathbbm{k}} G_i\cdot {\varepsilon}_U$ for all $(i,U)\in Q_G$, then $f({\varepsilon}_U)$ does
  belong to $(e_i*{\varepsilon}_U) \cdot (MG) \cdot (e_j*{\varepsilon}_V)$, which is
  contained in $\tilde e\cdot (T_S(M)*G)\cdot \tilde e$, for all
  arrows $f\colon (i,U)\to (j,V)$ of $Q_G$. Whence the
  isomorphism~(\ref{eq:27}).
\end{proof}

The proof of Theorem~\ref{sec:defin-main-results} uses Morita
equivalences in the graded sense. Two ($\mathbb N$-)graded finite
dimensional ${\mathbbm{k}}$-algebras $\Lambda_1$ and $\Lambda_2$ are Morita
equivalent in the graded sense if
$\Lambda_2\simeq \End_{\Lambda_1}(P)^{\mathrm{op}}$ for some graded
projective $\Lambda_1$-module $P$ which is a generator of the category
of graded $\Lambda_1$-modules, and hence of $\mathrm{mod}(\Lambda_1)$. Both
$\Lambda_1$ and $\Lambda_2$ are isomorphic as graded algebras to basic
graded finite dimensional algebras; moreover $\Lambda_1$ and
$\Lambda_2$ are Morita equivalent in the graded sense if and only if
their basic versions are isomorphic as graded algebras.

\begin{proof}[Proof of Theorem~\ref{sec:defin-main-results}]
  Denote by $\Phi$ the algebra homomorphism ${\mathbbm{k}} Q_G\to \intw^{\mathrm{op}}$ given in
  the statement of the theorem. Hence, $\Phi(\gamma)=f_\gamma$
  for all paths $\gamma$ in $Q_G$.

  (1) Let $\gamma$ and $\gamma'$ be parallel paths in $Q_G$. It
  follows from the definition of $\varphi_{\gamma'}$, see
  Notation~\ref{sec:defin-main-results-1}, and from (\ref{eq:22}) that
  \begin{equation}
    \label{eq:8}
    (f_\gamma|\varphi_{\gamma'}) =
    \left\{
    \begin{array}{ll}
      1 & \text{if $\gamma=\gamma'$,} \\
      0 & \text{otherwise.}
    \end{array}\right.
\end{equation}
Since $\Phi$ maps non parallel paths in $Q_G$ into distinct
components of $\intw$ in the decomposition (\ref{eq:26}), then $\Phi$
is injective.  In order to prove that $\Phi$ is an isomorphism, it is
hence sufficient to prove that, for all integers $n$, the homogeneous
components of degree $n$ of the graded algebras ${\mathbbm{k}} Q_G$ and $\intw$
have the same dimensions. Note that ${\mathbbm{k}} Q_G$ is graded by the length of
the paths in $Q_G$ and $\intw$ is graded so that every intertwiner
lying in a space of the shape $\Hom_{{\mathbbm{k}}
  G_{i_0}}(U,M(i_0,\ldots,i_n;V))$ is homogeneous of degree $n$.

First, following \cite{MR2578593}, see
Theorem~\ref{sec:proof-main-results-1},  the algebras ${\mathbbm{k}} Q_G$ and
$\tilde e\cdot (T_S(M)*G)\cdot \tilde e$ are isomorphic ${\mathbbm{k}}$-algebras
both Morita equivalent to $T_S(M)*G$. Actually, the proof given in
\cite{MR2578593} works in the graded setting.

Next, $\hat e$ may be considered
  as an idempotent of $T_S(M)*G$. In this sense,
  \[
  \tilde e\cdot (T_S(M)*G)\cdot \tilde e = \tilde e \cdot \hat e
  (T_S(M)*G) \cdot \hat e \cdot \tilde e\,.
  \]
  Following \cite{MR2578593}, see p.~1057, the algebra
  $\hat e \cdot (T_S(M)*G)\cdot \hat e$ is Morita equivalent to
  $T_{S_1}(\hat e \cdot(MG) \cdot \hat e)$. Again, this is still true
  in the graded sense. Hence,
  $\tilde e \cdot (T_S(M)*G)\cdot \tilde e$, which is basic, is
  isomorphic to $\tilde e \cdot T_{S_1}(\hat e \cdot (MG)\cdot \hat
  e)\cdot \tilde e$ as a graded algebra.

  Finally, by definition of $\tilde e$ and $\hat e$, the homogeneous
  component of degree $n$ of
  $\tilde e \cdot T_{S_1}(\hat e \cdot (MG) \cdot \hat e) \cdot \tilde
  e$ is equal to
  \[
  \bigoplus_{i_0,i_n\in [G\backslash Q_0],\,U_0\in \irr(G_{i_0}),\,U_n
    \in \irr(G_{i_n})}
  (e_{i_0}*{\varepsilon}_{U_0}) \cdot (MG)^{\otimes_{S_1}n} \cdot (e_{i_n}*{\varepsilon}_{U_n})\,.
  \]
  This is equal to
  \[
  \bigoplus_{i_0,\ldots,i_n\in [G\backslash Q_0],\,U_0\in \irr(G_{i_0}),\,U_n
    \in \irr(G_{i_n})}
  {\varepsilon}_{U_0}\cdot M(i_0,\ldots,i_n)\cdot {\varepsilon}_{U_n}\,.
  \]
  This is isomorphic to
  \[
  \bigoplus_{i_0,\ldots,i_n,U_0,U_n} \Hom_{{\mathbbm{k}} G_{i_0}}(U_0,
  \underset{\simeq
    M(i_0,\ldots,i_n;V)}{\underbrace{M(i_0,\ldots,i_n)\otimes_{G_{i_n}}
      V}})\,.
  \]

  These considerations prove that the homogeneous components of degree
  $n$ of ${\mathbbm{k}} Q_G$ and $\intw$ are isomorphic. Thus, $\Phi$ is an
  isomorphism.

  (2) follows from (\ref{eq:8}) and from the fact that $\Phi$ is
  surjective.
\end{proof}

\begin{proof}[Proof of Theorem~\ref{sec:defin-main-results-4}]
  First, $\Xi$ is bijective. Indeed, on one hand, there is an equality
  of vector subspaces of $T_S(M)*G$,
  \[
  \hat e \cdot (T_S(M) *G )\cdot \hat e
  =
  \bigoplus_{\underline i,\underline y}
  \left(
  \,_{i_0}M_{y_1\cdot i_1} \otimes_{{\mathbbm{k}}} \cdots \otimes_{{\mathbbm{k}}}
  \,_{y_1\cdots y_{n-1}\cdot i_{n-1}} M _{y_1\cdots y_{n}\cdot i_{n}}
  \right)
  * y_1\cdots y_n{\mathbbm{k}} G\,,
  \]
  where $\underline i$ runs through all sequences $i_0,\ldots,i_n$ of
  $[G\backslash Q_0]$ and $\underline y$ runs through all sequences
  $y_1,\ldots,y_n$ such that $y_t\in [G/G_{i_t}]$ for all
  $t\in \{1,\ldots,n\}$. On the other hand, for all vertices $(i,U)$
  and $(j,V)$ of $Q_G$ and for all ${\mathbbm{k}} G_i-{\mathbbm{k}} G_j$-bimodules $N$, the
  following mapping is bijective,
  \[
  \begin{array}{ccc}
    \Hom_{{\mathbbm{k}} G_i}(U,N\cdot {\varepsilon}_V) & \longrightarrow & {\varepsilon}_U \cdot N \cdot
                                                 {\varepsilon}_V \\
    f & \mapsto & f({\varepsilon}_U)\,.
  \end{array}
  \]
  Hence, $\Xi$ is bijective and $\Xi^{-1}$ is given by
  $\Xi^{-1}(f) = f({\varepsilon}_U)$ for all $f\in \intw$ with domain denoted by
  $U$.

  Next, $\Xi$ is a morphism of ${\mathbbm{k}}$-algebras. Indeed, for all
  $f\in \Hom_{{\mathbbm{k}} G_{i_0}}(U,M(\underline i;V))$ lying in $\intw$, then
  $\Xi^{-1}(f) = f({\varepsilon}_U)$.  Therefore, with the setting of
  Definition~\ref{sec:defin-notat-main-6}, then
  $(f'\circledast f) ({\varepsilon}_U)$ is equal to the product
  $f({\varepsilon}_V)\cdot f'({\varepsilon}_U)$ in $T_S(M)*G$, see (\ref{eq:36}). Accordingly,
  $\Xi^{-1}$, and hence $\Xi$, is a ${\mathbbm{k}}$-algebra homomorphism.

  Finally, the diagram~(\ref{eq:45}) is commutative. Indeed, all its
  arrows are ${\mathbbm{k}}$-algebra homomorphisms. Hence, it suffices to prove
  the commutativity on the arrows of $Q_G$, which follows from the
  definitions of $\Xi$, (\ref{eq:27}), and (\ref{eq:32}).  Note that
  (\ref{eq:44}) follows from (\ref{eq:28}).
\end{proof}

\section{Acknowledgements}

I thank Edson Ribeiro Alvares for helpful comments on previous
versions of this article. I also thank the referee for several
suggestions that improved the presentation of the article.

\bibliographystyle{abbrv}
\bibliography{biblio.bib}

\end{document}